\documentclass[10pt, english]{amsart}
\usepackage{amssymb,amsthm,amsmath,amstext,enumitem}
\usepackage{mathrsfs}

\usepackage[utf8]{inputenc}

\usepackage[left=2.6cm,top=2.6cm,right=2.6cm,bottom=2.6cm]{geometry}
\usepackage[all]{xy}
\usepackage{eucal}

\usepackage{hyperref}
\usepackage{comment}

\title{Logarithmic derivations of adjoint discriminants}

\author{Vladimiro Benedetti}

\author{Daniele Faenzi}

\author{Simone Marchesi}

\thanks{
D.F. and V.B. partially supported by FanoHK ANR-20-CE40-0023, SupToPhAG/EIPHI ANR-17-EURE-0002, Région Bourgogne-Franche-Comté, Feder Bourgogne and Bridges ANR-21-CE40-0017. S.M. partially supported by PID2020-113674GB-I00 and by the Spanish State Research Agency, through the Severo Ochoa and Mar\'ia de Maeztu Program for Centers and Units of Excellence in R\&D (CEX2020-001084-M).
}



\theoremstyle{plain}
\newtheorem{maintheorem}{Theorem}
\newtheorem{theorem}{Theorem}[section]
\newtheorem*{conjecture*}{Conjecture}
\newtheorem{proposition}[theorem]{Proposition}
\newtheorem{lemma}[theorem]{Lemma}
\newtheorem{corollary}[theorem]{Corollary}
\theoremstyle{definition}
\newtheorem{definition}[theorem]{Definition}
\newtheorem{remark}[theorem]{Remark}
\newtheorem{notation}[theorem]{Notation}
\newtheorem{example}[theorem]{Example}

\DeclareMathOperator{\Ext}{Ext}

\DeclareMathOperator{\IM}{Im}
\DeclareMathOperator{\Hom}{Hom}
\DeclareMathOperator{\rank}{rank}
\DeclareMathOperator{\coker}{coker}

\DeclareMathOperator{\length}{length}

\DeclareMathOperator{\Spec}{Spec}
\DeclareMathOperator{\Proj}{Proj}
\DeclareMathOperator{\Der}{Der}
\DeclareMathOperator{\Frac}{Frac}

\DeclareMathOperator{\Tot}{Tot}

\DeclareMathOperator{\rN}{N}
\DeclareMathOperator{\rs}{ss}

\def\Vv{V}

\def\bk{{\mathbf{k}}}
 
\def\bU{{\mathbf{U}}}
\def\bA{{\mathbf{A}}}
\def\bI{{\mathbf{I}}}
\def\bad{{\mathbf{ad}}}
\def\bR{{\mathbf{R}}}
\def\bF{{\mathbf{F}}}
\def\bS{{\mathbf{S}}}
 \def\II{{\mathbb{I}}}  \def\PP{{\mathbb{P}}}   \def\WW{{\mathbb{W}}} \def\ZZ{{\mathbb{Z}}} \def\VV{{\mathbb{V}}}

\def\SL{\mathrm{SL}}

\def\cE{{\mathcal{E}}}
\def\cF{{\mathcal{F}}}

\def\cI{{\mathcal{I}}}
\def\cJ{{\mathcal{J}}}

\def\cL{{\mathcal{L}}}
\def\cN{{\mathcal{N}}}
\def\cO{{\mathcal{O}}}

\def\cT{{\mathcal{T}}}
\def\cQ{{\mathcal{Q}}}
\def\cU{{\mathcal{U}}}

\def\cP{{\mathcal{P}}}
\def\fg{{\mathfrak g}}

\def\fsl{\mathfrak{sl}}
\def\fe{\mathfrak{e}}
\def\ff{\mathfrak{f}}
\def\fsp{\mathfrak{sp}}
\def\fso{\mathfrak{so}}
\def\fh{\mathfrak{h}}
\def\git{/\!\!/}

\def\GL{{\mathrm{GL}}}
\def\Sp{{\mathrm{Sp}}}
\def\SO{{\mathrm{SO}}}
\def\O{{\mathrm{O}}}

\def\Proj{{\mathrm{Proj}}}
\def\log{{\mathrm{log}}}
\def\sm{{\mathbf{sm}}}

\def\iff{\Leftrightarrow}
\def\TD{\cT_{\check \PP}\langle D \rangle}
\def\pip{\pi}
\def\rhop{\rho}

\usepackage[usenames,dvipsnames]{xcolor}

\begin{document}

\sloppy

\begin{abstract}
We exhibit a relationship between projective duality and the sheaf of logarithmic vector fields along a reduced divisor $D$ of projective space, in that the push-forward of the ideal sheaf of the conormal variety in the point-hyperplane incidence, twisted by the tautological ample line bundle is isomorphic to logarithmic differentials along $D$.

Then we focus on the adjoint discriminant $D$ of a simple Lie group with Lie algebra $\fg$ over an algebraically closed field $\bk$ of characteristic zero and study the logarithmic module $\Der_\bU(-\log(D))$ over $\bU=\bk[\fg]$. 
When $\fg$ is simply laced, we show that this module has two direct summands: the $G$-invariant part, which is free with generators in degrees equal to the exponents of $G$, and the $G$-variant part, which is of projective dimension one, presented by the Jacobian matrix of the basic invariants of $G$ and isomorphic to the image of the map $\bad : \fg \otimes \bU(-1) \to \fg \otimes \bU$ given by the Lie bracket.

When $\fg$ is not simply laced, we give a length-one equivariant graded free resolution of $\Der_\bU(-\log(D))$ in terms of the exponents of $G$ and of the quasi-minuscule representation of $G$.
\end{abstract}

\maketitle

\section*{Introduction}

Consider a reduced hypersurface $D=\VV(F)$ of the projective space $\PP$ over an algebraically closed field $\bk$ of characteristic $0$, defined by a homogeneous square-free form $F$ in the polynomial coordinate ring $\bU$ of $\PP$.
The sheaf of logarithmic derivations $\cT_{\PP}\langle D \rangle$ is the dual of Deligne's sheaf of 1-forms on $\PP$ with logarithmic poles along $D$ (see \cite{deligne:hodge}). This sheaf was first studied by Saito (see \cite{saito:logarithmic}) in connection with discriminants of simple singularities and later on it has found rather diverse applications, for instance in the theory of arrangements and free divisors (starting from \cite{terao:freeness-I, terao:freeness-II, Terao}, for an overview see  \cite{orlik-terao:arrangements, dimca:arrangements}), unfolding of singularities (see \cite{buchweitz-ebeling-von_bothmer,damon:legacy-I}) and locally trivial deformations (see \cite{sernesi:deformations}).

One of the main questions about the sheaf of logarithmic derivations or its associated graded module of global sections  $\Der_\bU(-\log(F))$,
is to compute the graded Betti numbers of this module. 
As it turns out, information on the projective dimension or on the number and minimal degree of minimal generators is in general quite hard to acquire. Even the freeness of this module is the object of an important open question by Terao (see \cite{Terao}) to the effect that, when $D$ is a hyperplane arrangement, freeness depends only on the isomorphism class of the intersection lattice of the hyperplanes. 
Freeness holds for some important classes of arrangements connected with Weyl groups of simple Lie algebras, see for instance \cite{yoshinaga:characterization}. However, in spite of recent developments (see \cite{dipasquale:characterization}), even for these arrangements our knowledge of Betti numbers beyond the free range is very partial, see \cite[Conjecture 1]{abe-vaenzi-valles} and \cite{abe-faenzi} for recent developments.

From a different perspective, a very powerful tool to compute graded free resolutions is provided a method based on resolutions of singularities arising from Kempf collapsing as introduced in \cite{kempf:collapsing}; we refer for instance to \cite{lascoux:determinantal} for applications to the free resolution of determinantal ideals. We refer to \cite{Weyman} for a complete presentation of this method and of its many applications. This has been studied with particular emphasis in the framework of projective duality in \cite{gelfand-kapranov-zelevinsky}. However, to our knowledge, this method has not been applied so far to logarithmic derivations.

The first main result of the present paper outlines a simple and deep connection between logarithmic derivations and projective duality. This is the first step in order to apply the geometric method to the computation of Betti numbers of modules of logarithmic derivations. 

To explain this connection, we think of $D$ as sitting in the dual space $\check \PP$ of hyperplanes of $\PP$ and consider the projective dual $X$ of $D$. The \textit{conormal} variety $\WW_X$ of $X$, namely the blow-up of $D$ along its Jacobian scheme, sits in the point-hyperplane incidence $\II \subset \PP \times \check \PP$. 
Let $\tilde \rhop : \II \to \check \PP$ be the projection onto the second factor.

\begin{maintheorem} \label{main:direct image}
Let $h$ be the pull back to $\II$ of the hyperplane class of $X$.
Then: $$
    \TD(-1) \simeq \tilde{\rhop}_* \left(\cI_{\WW_X/\II}(h)\right).
$$
\end{maintheorem}

This provides a wide generalisation of the approach developed in
\cite{FaenziValles, faenzi-matei-valles}, where the sheaf of logarithmic derivations of a hyperplane arrangement was computed as Fourier-Mukai transform of $\cI_Z(1)$, via the point-hyperplane incidence correspondence $\II$, where $Z$ is the set of points in the dual space $\check \PP$ corresponding to the arrangement and 
$\cI_Z$ is the ideal sheaf of $Z$ in $\check \PP$.
This result was in some sense at the origin of the study of unexpected curves, carried on by many authors starting with \cite{cook-harbourne-migliore-nagel}.

Our main application is to the minimal resolution of logarithmic derivations modules of adjoint discriminants.
Let $G$ be a simple linear algebraic group over $\bk$ and set $\fg$ for the Lie algebra of $G$. Let $n$ be the rank of $\fg$ and $e_1 \le \cdots \le e_n$ be the exponents of $\fg$.
Let $\bU=\bk[\fg]$ and, denoting by $(x_1,\ldots,x_\ell)$ the coordinates for $\bk[\fg]$, we write $\Der_\bU = \bU\partial_{x_1} \oplus \cdots \oplus \bU\partial_{x_\ell}$. 
The fundamental divisor studied in this paper is $D=\VV(\Delta)$, where $\Delta$ is the adjoint discriminant, i.e., the discriminant for the adjoint action of $G$ on $\fg$. 
We are mainly interested in the module of logarithmic derivations $\Der_\bU(-\log(\Delta))$ defined by the adjoint discriminant. 
Let us first state our main result for $G$ of type $A$, $D$, $E$, i.e., when $\fg$ is \textit{of simply laced type}.
To do this, we need to introduce the following two main ingredients. The first one is the natural morphism:
$\bad : \fg\otimes \bU(-1) \to \fg\otimes \bU$ 
given by the dual of the Lie bracket on the $1$-degree parts $(\fg\otimes \bU(-1))_1 = \fg \to \fg \otimes \fg = (\fg \otimes \bU)_1$ and extended by $\bU$-linearity (here we identify $\fg$ and $\fg^\vee$ via the Killing form).
The $\bU$-modules $\ker(\bad)$ and $\bA = \IM(\bad)$ are $G$-equivariant; moreover, it will turn out that $\ker(\bad)$ is a free module, more precisely $\ker(\bad) \simeq \bigoplus_{i=1}^n \bU(-e_i-1)$.

The second ingredient is the \textit{braid (or Weyl) arrangement} of type $G$ obtained by intersecting $\Delta$ with a Cartan subalgebra $\fh \subset \fg$. The Weyl group $W$ of $G$ acts on $\fh$ by reflections about the hyperplanes $H_\alpha$, where $\alpha$ ranges in the set $\Phi^+$ of positive roots of the root system $\Phi$ of $\fg$.

A fundamental fact is that $\Delta|_\fh = \delta$, with $\sqrt{\delta}= \Pi_{\alpha \in \Phi^+} \delta_\alpha$, where $\delta_\alpha$ is a linear equation defining the hyperplane $H_\alpha$.
The \textit{remarkable formula of Kostant, Macdonald,
Shapiro and Steinberg} (KMSS), referring to work of Shapiro and \cite{kostant:principal, macdonald, steinberg:finite}, asserts that the braid arrangement is free with generators sitting in degrees equal to the exponents. This is related to the topology of the complement in $\mathbb{C}^n$ of the arrangement, which is a $K(\pi,1)$ space of the corresponding Artin-Tits group of type $G$, see \cite{fox-neuwirth}.
Terao showed in \cite{Terao} that the Poincaré polynomial of the complement of a free central affine arrangement with generators of degree $(d_1,\ldots,d_n)$ equals $\Pi_{i=1}^n(1+td_i)$, recovering a formula of Orlik and Solomon (see \cite{orlik-solomon:unitary}). 
A free basis of $\Der_\bS(-\log(\delta))$ is given by the \textit{Saito matrix}, which can be explicitly described as follows (see e.g. \cite{Yos, saito:uniformization}). Writing $(f_1,\ldots,f_n) \in \bS^n$ for a free basis of $\bS^W$, with $\deg(f_i)=e_i$ for $i=1,\ldots,n$, one has
 \[
 \Der_\bS(-\log(\delta)) \simeq \mu_1\bS \oplus \cdots  \oplus \mu_n\bS, \:\: \mbox{ with: } \mu_i=\nabla(f_i).
 \]

Going back to the full Lie algebra $\fg$, the $W$-invariants $f_1,\ldots,f_n$ lift to $G$-invariants $F_1,\ldots,F_n \in \bU$ via the Chevalley restriction theorem, which implies that $\bU^G \simeq \bS^W$. 
We get a new Saito-type matrix
\[
 \nu=(\nu_1,\ldots,\nu_n) : \bigoplus_{i=1}^n \bU(-e_i) \longrightarrow \fg\otimes \bU, \:\: \mbox{ with: }  
\nu_i=\nabla(F_i).
\]
Our result brings these objects together in the following statement.

\begin{maintheorem}\label{main simply laced}
Let $G$ be a simple linear algebraic group over $\bk$ and let $\fg$ be its Lie algebra. Assume $\fg$ is of simply laced type, let $e_1,\ldots,e_n$ be the exponents of $\fg$ and let $\Delta$ be the adjoint discriminant. Then:
\[
\Der_\bU(-\log(\Delta)) \simeq \bigoplus_{i=1}^n \bU(-e_i) \oplus \bA,
\]
where the module $\bA$ is the image of $\bad : \fg\otimes \bU(-1) \to \fg\otimes \bU $ and fits into:
\[
0\to \bigoplus_{i=1}^n \bU(-e_i-1) \xrightarrow{\nu} \fg\otimes \bU(-1) \to \bA \to 0.\]
\end{maintheorem}

In particular, the module $\Der_\bU(-\log(\Delta))$ is of projective dimension one, with a summand presented by a \textit{rectangular Saito matrix} corresponding to the image of the Lie bracket, and a free summand with the same exponents as the corresponding Weyl arrangement.

Let us next discuss simple Lie algebras which are not of simply laced type. 
In this case, the connection with the KMSS formula is a bit different, as the adjoint discriminant captures long positive roots, rather than all of $\Phi^+$, though the discussion might be related to ideal arrangements as in \cite{abe-barakat-cuntz-hoge-terao}.
Anyway, we give a minimal graded free equivariant resolution of the derivation module also in this case and this module turns out to have again projective dimension one. Let us denote by $s$ the number of long simple roots of $\fg$. The new ingredient is the quasi-minuscule representation $\hat \fg$ of $\fg$, namely, the irreducible representation whose highest weight is the highest short root (as opposed to the adjoint representation, whose highest weight is the highest long root).
Let $j$ be the smallest short exponent of $G$, see \cite{kostant:lie-group}. Concretely, we have $j=1,2,3,n-1$ when $G$ is of type, respectively, $C_n$, $G_2$, $F_4$ and $B_n$.
Let $s$ be the number of long simple roots. 
Let us state our last main result. 

\begin{maintheorem} \label{main not simply laced} Assume $G$ is a simple linear algebraic group over $\bk$, not of simply laced type, $\Delta$ its adjoint discriminant. Then the module  $\Der_\bU(-\log(\Delta))$
admits a $G$-equivariant graded free resolution of the form:
\[
0 \to \begin{array}{c}
\bigoplus_{i=1}^s \bU(-e_i-1) \\ \oplus \\ \hat \fg \otimes \bU(-j-1)
\end{array}\to
\begin{array}{c}
 \bigoplus_{i=1}^s \bU(-e_i) \\
 \oplus \\
   \fg \otimes \bU(-1) \oplus \hat \fg \otimes \bU(-j)
 \end{array}
 \to \Der_\bU(-\log(\Delta)) \to 0.
\]
\end{maintheorem}

The strategy of the proof of Theorems \ref{main simply laced} and \ref{main not simply laced} goes as follows. We rely on Theorem \ref{main:direct image} and use the geometric technique mentioned above to compute $\tilde{\rhop}_*(\cI_{\WW_X/\II}(h))$ via the pushforward of a Koszul complex on $X\times \check{\PP}$. This translates into computing the cohomology of twisted exterior powers of the affine tangent sheaf of $X$. Here, we use the geometry of adjoint varieties and notably the contact structure to identify the affine tangent bundle with the affine cotangent bundle on $X$, up to a modification of twists. The computation of the cohomology of these bundles turns out to be affordable, on a case-by-case basis, through a quite technical use of the Bott-Borel-Weil theorem and branching rules for classical groups. However, rather surprisingly, the final result can be stated in a uniform and concise way for all groups, regardless of whether $G$ is of simply laced type or not (see Theorem \ref{thm_res}). Of course, a unified proof would be highly desirable. As a final step, for simply laced types, we connect our discussion to the theory of free arrangements to prove that the resolution is minimal and that it induces the decomposition stated in Theorem \ref{main simply laced}.
\medskip

The paper is organised as follows. In \S \ref{section : pushforward} we develop our study of logarithmic tangent sheaves via pushforward from the normal variety and prove Theorem \ref{main:direct image}.
In \S \ref{sec-Weymanmethod} we start working out resolutions of the Jacobian ideal and of the logarithmic tangent sheaf, obtained using the Cayley method and pushforward in the spirit of \cite{gelfand-kapranov-zelevinsky}.
We recall several results on the normalisation of dual varieties, particularly for Fano varieties, (see Proposition~\ref{normalisation}) and we describe the relationship between the geometric method and the Jacobian ideal (see Theorem~\ref{thm_resolution_pushforward}).
In \S \ref{section : adjoint} we review some of the relevant geometry of adjoint varieties and their contact structure. Moreover, we provide a minimal graded free resolution of the normalisation of adjoint discriminants (see Theorem \ref{thm_res_structure}).
In \S \ref{section: adjoint discriminants}, we compute a resolution of the module of logarithmic derivations of adjoint discriminants (see Theorem \ref{thm_res}) both for the simply laced and the non-simply laced situation. We also recover some results of
\cite{faenzi-marchesi} regarding  symmetric determinants in terms of the adjoint orbits in the case $C_n$.
Finally, in \S \ref{section: non invariant}, using Terao's results on hyperplane arrangements, we explicitly construct the module of $G$-invariant logarithmic derivations and deduce Theorem \ref{main simply laced}.

We would like to thank D. Fratila and M. Yoshinaga for fruitful discussions.

\section{Jacobian ideal and logarithmic differential via pushforward}
\label{section : pushforward}

The main goal of this section is to introduce the use of projective duality in the analysis of the Jacobian ideal of a hypersurface $D$ in the projective space, as well as the module of logarithmic derivations along $D$, leading to Theorem \ref{thm-pushforwardlog}, which is essentially Theorem \ref{main:direct image} from the introduction.
We first set up some background material in \S \ref{subsection:background}, where we review different notions of affine/projective version of logarithmic derivations modules and sheaves, as well as some conventions about projective duality.
Then, in \S \ref{subsection:dual} we state and prove Theorem \ref{thm-pushforwardlog}, as well as some consequences and connections with previous results.

\subsection{Background and notation}

\label{subsection:background}

We work over an algebraically closed field $\bk$ of characteristic zero. Let $\ell \ge 1$ be an integer.
Let us denote by $\PP=\PP^{\ell-1}$ the $\ell-1$-dimensional projective space parametrising hyperplanes in $\bk^\ell$ and by $\check{\PP}=\check{\PP}^{\ell-1}$ the dual projective space parametrising hyperplanes of $\PP$.

\subsubsection{Logarithmic derivations}

Let $\bU=\bk[x_1,\ldots,x_\ell]$ be the coordinate ring of $\check \PP^{\ell-1}$. We define the module of $\bU$-derivations as the free $\bU$-module of rank $\ell$:
\[
\Der_\bU = \left\{\sum_{i=0}^\ell \theta_i \partial_i \mid \theta_i \in \bU\right\}, \qquad \mbox {with:} \qquad \partial_i = \frac {\partial}{\partial x_i}, \qquad \text{for $i \in \{0,\ldots,\ell\}.$}
\]

\begin{definition}
Let $F \in \bU$ be a non-zero square-free homogeneous polynomial of degree $d$. Then we set:
\begin{align*}
\Der_\bU(-\log (F))& = \{\theta \in \Der_\bU \mid \theta(F) \in (F)\}, \\
\Der_\bU(-\log (F))_0&= \{\theta \in \Der_\bU \mid \theta(F) =0 \}.
\end{align*}
The $\bU$-module of \textit{logarithmic derivations} $\Der_\bU(-\log (F))$ and its submodule
$\Der_\bU(-\log (F))_0$ are $\ZZ$-graded, respectively of rank $\ell$ and $\ell$. 
\end{definition}

The Euler derivation $\epsilon = x_1\partial_0+\cdots+x_\ell \partial_\ell$ provides a splitting:
\[
\Der_\bU(-\log (F)) = \Der_\bU(-\log (F))_0 \oplus \bU \epsilon.
\]

\subsubsection{Logarithmic tangent sheaves}

For this part we mainly refer to \cite[\S 3.4.4]{sernesi:deformations} but, for instance,
useful properties of the logarithmic tangent sheaf can also be found in \cite{dolgachev:logarithmic}.
Let $Y$ be a reduced projective $\bk$-scheme and write $\cT_Y$ for the tangent sheaf of $Y$.
Assume now that $Y$ lies in $\PP=\PP^{\ell-1}$ with normal sheaf $\cN_{Y/\PP}$ and consider the composition:
\[
\nabla : \cT_{\PP} \to \cT_{\PP}|_Y \to \cN_{Y/\PP}.
\]

\begin{definition}
The sheaf of \textit{logarithmic differentials} $\cT_\PP\langle Y \rangle$ or \textit{logarithmic tangent sheaf} is the kernel of $\nabla$ and the equisingular normal sheaf $\cN'_{Y/\PP}$ is the image of $\nabla$, so we have an exact sequence:
\begin{equation} \label{TY}
0 \to \cT_\PP\langle Y \rangle \to \cT_{\PP} \to \cN'_{Y/\PP} \to 0.
\end{equation}
The \textit{equisingular normal sheaf} $\cN'_{Y/\PP}$ is supported on $Y$ and its rank along an irreducible component of $Y$ is equal to the codimension  of this component. 
\end{definition}

If $Y$ is a reduced hypersurface $D \subset \check{\PP}$ defined by a single homogeneous equation $F \in \bU$ of degree $d$, then:
\[\cN_{D/\check \PP} \simeq \cO_D(d), \qquad
\cN'_{D/\check \PP} \simeq \cJ_D(d),\]
where $\cJ_D$ is the restricted Jacobian ideal sheaf, defined as the restriction to $D$ of the Jacobian ideal sheaf $\cJ$ generated by the partial derivatives of $F$, namely:
\[
\cJ(d) = \IM(\nabla(F))=\IM\left((\partial_0 F,\ldots,\partial_\ell F) : \cO_{\check \PP}^{\ell +1}(1) \to \cO_{\check \PP}(d)\right).
\]
Rewriting \eqref{TY}, we get the fundamental exact sequence:
\begin{equation}\label{diag-deglogtang}
0 \rightarrow \TD \rightarrow \cT_{\check{\PP}} \rightarrow \cJ_D(d) \rightarrow 0.
\end{equation}

This allows us to consider $\TD$ as the sheaf of Jacobian syzygies of $F$, namely the kernel of $\nabla(F)$, in view of the next exact diagram. Note that commutativity of the top right square follows from the Euler relation.
\begin{equation}\label{diag-logsyz}
\xymatrix@-1.5ex{
& & \cO_{\check{\PP}}(-1) \ar@{^(->}[d] \ar^-{\cdot d}[r] & \cO_{\check{\PP}}(-1) \ar@{^(->}^{\cdot F}[d]\\
0 \ar[r] &\TD(-1) \ar[r] \ar@{=}[d] & \cO_{\check{\PP}}^{\ell} \ar[r]^-{\nabla F} \ar@{->>}[d] & \cJ(d-1) \ar[r] \ar@{->>}[d] & 0 \\
0 \ar[r] & \TD(-1) \ar[r] & \cT_{\check{\PP}}(-1) \ar[r] & \cJ_D(d-1) \ar[r] & 0 \\
 }
\end{equation}

We have the fundamental relationships involving the sheafified module of logarithmic derivations and the graded module of logarithmic differentials, valid for a reduced hypersurface $D =\VV(F)\subset \check \PP$
\begin{align}
     \label{relation-T-Der-0}
\Der_\bU(-\log (F))_0 &\simeq \bigoplus_{t \in \ZZ} H^0(\TD(t-1)), && \widetilde{\Der_\bU(-\log (F))_0} \, \simeq \TD(-1), \\
     \label{relation-T-Der}
\Der_\bU(-\log (F)) &\simeq \bigoplus_{t \in \ZZ} H^0(\hat \cT_{\check \PP}\langle D \rangle(t-1)), && \widetilde{\Der_\bU(-\log (F))} \, \simeq \hat \cT_{\check \PP}\langle D \rangle (-1).
\end{align}

\subsubsection{The affine logarithmic tangent sheaf}

Let $Y$ be a reduced closed subscheme of $\PP$ and consider the induced hyperplane bundle $c_1(\cO_Y(1))$.
The first Chern class of $\cO_Y(1)$ provides a non-zero element of $H^1(Y,\Omega_Y)$ and, in turn, a non-splitting extension 
$$ 0\to \Omega_Y \to \hat{\Omega}_Y \to \cO_Y \to 0,$$
where the middle sheaf $\hat{\Omega}_Y$ defined by the sequence is called the \textit{affine cotangent sheaf} of $Y$.
Its dual is the \textit{affine tangent sheaf} $\hat \cT_Y$ of $Y$, fitting into:
\begin{equation} \label{affine tangent}
    0\to \cO_Y \to \hat{\cT}_Y \to \cT_Y \to 0.
\end{equation} 

Note that the definition of $\hat{\cT}_Y$ depends on the inclusion $Y\subset \PP$ and that the affine tangent sheaf of the projective space is $\hat{\cT}_\PP=V\otimes \cO_\PP(1) \cong \cO_\PP(1)^{\ell}$. Moreover, we have an exact sequence 
$$ 0\to \hat{\cT}_Y \to \hat{\cT}_\PP|_Y \to \cN_{Y/\PP},$$
whose image is again the equisingular normal sheaf $\cN'_{Y/\PP}$.
\begin{definition}
The \textit{affine sheaf of logarithmic differentials} 
$\hat \cT_\PP\langle Y \rangle$ of $Y$ is the kernel of the natural composition:
\[
\hat{\cT}_\PP \to \hat{\cT}_\PP|_Y \to \cN_{Y/\PP}. 
\]
We have exact sequences:
\begin{align} \label{hatTY}
&0 \to \hat \cT_\PP\langle Y \rangle \to \hat \cT_{\PP} \to \cN'_{Y/\PP} \to 0,\\
\nonumber &0 \to \cO_{\PP} \to \hat \cT_\PP\langle Y \rangle \to \cT_\PP\langle Y \rangle \to 0.
\end{align}

When $Y=D$ is a hypersurface of $\check \PP$, \eqref{hatTY} reads as follows:
    \begin{equation}\label{eq_affine_tangent_jac} 
    0\to \hat \cT_\PP\langle D \rangle \to \hat{\cT}_{\check{\PP}} \to \cJ_D(d) \to 0.
    \end{equation}
In terms of $\bU$-modules, we consider the affine cone $\hat D$ of the hypersurface $D \subset \check \PP$ defined by a homogeneous equation $F$. We let $J_D$ be the restricted Jacobian ideal of $F$, namely the ideal generated by the partial derivatives of $F$ inside $\bU/(F)$.
Then, using \eqref{relation-T-Der}, the sequence \eqref{eq_affine_tangent_jac} gives:
    \begin{equation}\label{eq_affine_tangent_jac_modules} 
    0\to \Der_\bU(-\log(F)) \to \Der_\bU \to J_D(d-1) \to 0.
    \end{equation}
\end{definition}

We will recover this sequence via pushforward in Corollary \ref{cor_affine_tanlog}.
Next, we observe a basic fact, that explains the relationship 
between two natural ways of defining the logarithmic tangent sheaf, namely as a quotient of the affine logarithmic tangent sheaf or, up to twist, as the sheafification of $\Der_\bU(-\log (F))_0$.

\begin{lemma}
The affine sheaf of logarithmic differentials of a hypersurface $D$ splits as $ \hat \cT_\PP\langle D \rangle \simeq  \cT_\PP\langle D \rangle \oplus \cO_{\PP}$.
\end{lemma}
\begin{proof}
Recall that we denote by $\cJ_D$ the restriction of the Jacobian ideal to $D$. Diagram (\ref{diag-logsyz}) implies that, applying the functor $\Hom(-,\cO_{\PP})$ to its lower horizontal sequence (corresponding to the sequence in \eqref{eq_affine_tangent_jac}), the induced map
$$
\Ext^1\left(\cJ_D(d),\cO_\PP\right) \rightarrow \Ext^1\left(\cT_\PP,\cO_\PP\right)\simeq \bk
$$
is surjective. This implies that the following map between extension groups vanishes:
\begin{equation}\label{diag-logext}
\Ext^1\left(\cT_\PP,\cO_\PP\right) \rightarrow \Ext^1\left(\cT_\PP\langle D \rangle,\cO_\PP\right).
\end{equation}
Notice that the extension in (\ref{hatTY}) is defined between the logarithmic derivations and the Euler vector field, which implies the following commutative diagram
$$
\xymatrix@-1ex{
 & \cO_{\PP} \ar@{^(->}[d] \ar@{=}[r] & \cO_{\PP} \ar@{^(->}[d] \\
 0 \ar[r] & \hat \cT_\PP\langle D \rangle \ar@{->>}[d] \ar[r] & \hat \cT_\PP \ar@{->>}[d] \ar[r] & \cJ_D(d) \ar[d]^\simeq \ar[r] & 0 \\
 0 \ar[r] & \cT_\PP\langle D \rangle  \ar[r] &  \cT_\PP \ar[r] & \cJ_D(d) \ar[r] & 0\\
 }
$$
This allows us to conclude that the extension class considered in (\ref{hatTY}) belongs to the image of (\ref{diag-logext}) and therefore it is the zero class. This proves the statement. 
\end{proof}

\begin{remark}
The previous discussion also applies over fields of positive characteristic, provided that $\mathrm{char}(\bk)$ does not divide $d$. This restriction is necessary because when $\mathrm{char}(\bk)$ does divide $d$, we have that the Euler derivation belongs to $\Der_\bU(-\log (F))_0$, inducing a different diagram with respect to Diagram \eqref{diag-logsyz}, from which we cannot have the described splitting. In any case, $\cT_\PP\langle D \rangle$ can always be defined as the sheafification of the quotient of $\Der_\bU(-\log(F))$ by the Euler derivation, but this no longer identifies with $\Der_\bU(-\log (F))_0$.
\end{remark}

\subsubsection{Projective duality}

Consider a reduced closed subscheme $X$ of $\PP=\PP^{\ell-1}$ and write $X_{\sm}$ for its smooth locus. Set $\cT_{X,x}$ for the tangent space of $X$ at $x\in X$. Then we define the normal variety $\WW_X$ as the Zariski closure in $\PP \times \check \PP$ of:
$$
\WW^\circ_X = \left\{ (x,H) \in X_{\sm} \times \check{\PP} \: | \: H \supset \cT_{X,x} \right\}.
$$
In particular, if $X$ is smooth, it is possible to define $\WW_X$ as $\PP(\cN_{X/\PP}(-1))$.

Observe that $\WW_X$ is a reduced closed subscheme of the point-hyperplane incidence variety $\II$:
$$
\WW_X \subset \II= \left\{ (x,H) \in \PP \times \check{\PP} \: | \: x \in H \right\}.
$$

Denote by $h$ and $\check{h}$, respectively, the hyperplane classes of $\PP$ and $\check \PP$. Denote by $\hat{\pip}$ and $\hat{\rhop}$ the canonical projections of the product $\PP \times \check{\PP}$ to each of its factors. Their restrictions to $X \times \check{\PP}$ will be denoted respectively by $\overline{\pip}$ and $\overline{\rhop}$ and, furthermore, we will denote by $\pip$ and $\rhop$ their respective restrictions to $\WW_X$. By abuse of notation, we will denote also by $h$ and $\check{h}$ their respective pullbacks to $\PP \times \check{\PP}$, as well as their restrictions to the subvarieties described above.

The image $\check{X}:=\overline{\rhop}(\WW_X)$ is known as the \textit{dual variety} of $X$.
The \textit{biduality theorem} asserts that $X$ itself is the dual of $\check{X}$. 
The crux of the biduality theorem is the fact that $\WW_X = \WW_{\check{X}}$ as subvarieties of $\PP \times \check{\PP}$. 
We refer to \cite{gelfand-kapranov-zelevinsky} for a comprehensive description. Finally, we will denote by $\tilde{\pip}$ and $\tilde{\rhop}$ the two projections from $\II$, respectively to $\PP$ and $\check{\PP}$, given once more as the restrictions from $\PP \times \check{\PP}$. We will denote by the same notation the corresponding affine projections. For instance, if $\check{\PP}=\PP(V^\vee)$ then the projections from $X\times V$ to $X$ and $V$ will be denoted again, respectively, by $\overline{\pip}$ and $\overline{\rhop}$.

\subsection{Logarithmic differentials as direct image sheaves}
\label{subsection:dual}

We are now in the position to state and prove Theorem \ref{main:direct image}, the main result of this section.
\begin{theorem}\label{thm-pushforwardlog}
    Let $X \subset \PP$ be a reduced closed subscheme such that all the irreducible components of $D = \check{X}$ are hypersurfaces of $\check{\PP}$. Then the logarithmic tangent sheaf $\TD$ of $D$ satisfies
    $$
    \TD (-1) \simeq \tilde{\rhop}_* \left(\cI_{\WW_X/\II}(h)\right),
    $$
    where $\cI_{\WW_X/\II}$ is ideal sheaf of $\WW_X$ seen as a subscheme of the flag variety $\II$.
\end{theorem}

\begin{proof}
    Let $d= \deg(D)$.
    The normal variety $\WW_X$ is the blow-up of $D$ along the Jacobian subscheme, defined by the ideal sheaf $\cJ_D\subset \cO_D$, see for instance \cite[Theorem 1]{Nobile}. In turn, the blow-up of $D$ along $\cJ_D$ is the projectivisation of the Rees algebra $\bigoplus_{t \ge 0} \cJ_D^t$ of $\cJ_D$, which is the quotient of the symmetric algebra of $\cJ_D$ by its torsion part. The projectivisation of $\cJ_D$  defines the scheme $\PP(\cJ_D) \simeq \PP(\cJ_D(d-1))$ and taking the torsion-free part corresponds to isolating the irreducible components of $\PP(\cJ_D)$ forming $\WW_X$, the remaining components (called the torsion components) being projected onto the locus in $D$ where the Jacobian subscheme fails to be locally a complete intersection (see \cite{bigna2018} for a comprehensive discussion). Using the surjection $\cT_{\check{\PP}}(-1) \to \cJ_D(d-1)$ given by the derivatives of the polynomial $F$ defining $D$ (see \eqref{diag-deglogtang}), we write epimorphisms
    \[\bigoplus_{t \ge 0} S^t(\cT_{\check \PP}(-1)) \to \bigoplus_{t \ge 0} S^t(\cJ_D(d-1)) \to \bigoplus_{t \ge 0} (\cJ_D(d-1))^t.\]
    Taking $\Proj$ gives rise to a canonical
     chain of inclusions 
    \begin{equation}\label{norm-in-flag}
    \WW_X \subset \PP(\cJ_D(d-1)) \subset \PP(\cT_{\check{\PP}}(-1))=\II.
    \end{equation}
    
    Moreover, the tautological line bundles associated to the $\Proj$ constructions above agree upon restriction, namely 
    $\cO_{\WW_X}(h)$ is the relatively ample line bundle given by the blow-up construction and agrees with the restriction of $\pi^*(\cO_X(1))$. Pushing forward via $\rho$ the map  $\rho^* (\cJ_D(d-1)) \to \cO_{\WW_X}(h)$, obtained by the Proj construction of the blow-up, we get
    a canonical inclusion of the restricted Jacobian ideal
    \begin{equation}\label{diag-incJacob}
    \cJ_D(d-1) \hookrightarrow \rhop_*(\cO_{\WW_X}(h)) \simeq \rhop_*(\pip^*(\cO_X(1))).
    \end{equation}
   Consider the following short exact sequence
    $$
    0 \longrightarrow \cI_{\WW_X/\II}(h) \longrightarrow \cO_\II(h) \longrightarrow \cO_{\WW_X}(h) \longrightarrow 0.
    $$
    Taking its pushforward by $\tilde{\rhop}$, we get the bottom row of the following commutative diagram
    \begin{equation}\label{diag-logaritmictang}
\xymatrix{
0 \ar[r] & \TD(-1) \ar[d] \ar[r] & \cT_{\check{\PP}}(-1) \ar[d]_{\simeq} \ar[r] & \cJ_D(d-1)\ar@{^{(}->}[d] \ar[r] & 0\\
0 \ar[r] & \tilde{\rhop}_*(\cI_{\WW_X/\II}(h)) \ar[r] & \tilde{\rhop}_*(\cO_\II (h)) \ar[r] & \rhop_*(\pip^*(\cO_X(1))) & & .
}
\end{equation}

Indeed, the middle vertical map is an isomorphism as a direct consequence of the definition of $\cO_{\WW_X}(h) \simeq \pip^*(\cO_X(1))$.
Furthermore, notice that the commutativity of the previous diagram (in particular of the right square) follows naturally from the description of the normal variety as a subvariety of the flag $\II$, passing through the projectivisation of the Jacobian ideal, as depicted in (\ref{norm-in-flag}).
Finally, from the snake lemma applied to the previous diagram, we find the desired isomorphism
\begin{equation*}
     \TD(-1) = \tilde{\rhop}_* \left(\cI_{\WW_X/\II}(h)\right). \qedhere  
\end{equation*}
\end{proof}

Let us now state and comment some direct consequences of the previous theorem.
The first result will relate the logarithmic sheaf to the ideal sheaf of $\WW_X$ seen this time as a subvariety of the product $X \times \check{\PP}$. The described relation will be fundamental in Section \ref{sec-Weymanmethod}.

\begin{corollary}
\label{cor_affine_tanlog}
    Let $X$ be as in Theorem \ref{thm-pushforwardlog} and, moreover, assume it to be non-degenerate and linearly normal. Then we have the following short exact sequence
    $$
    0 \longrightarrow \cO_{\check{\PP}}(-1) \longrightarrow \overline{\rhop}_*(\cI_{\WW_X/(X\times \check{\PP})}(h)) \longrightarrow \TD(-1) \longrightarrow 0.
    $$
\end{corollary}
\begin{proof}
    Consider the following commutative diagram
    \begin{equation*}
    \xymatrix@-1.5ex{
    & 0 \ar[d] & 0 \ar[d] & 0 \ar[d] \\
    0 \ar[r] & \mathcal{K} \ar[r] \ar[d] & \hat{\pip}^* \cI_X(1) \otimes \hat{\rhop}^*\cO_{\check{\PP}} \ar[r] \ar[d] &  \cI_{\WW_X/\II}(h) \ar[r] \ar[d] & 0\\
    0 \ar[r] & \cO_{\PP \times \check{\PP}}(-\check{h}) \ar[r] \ar[d] & \cO_{\PP \times \check{\PP}}(h) \ar[r] \ar[d] & \cO_\II (h) \ar[r] \ar[d] & 0\\
    0 \ar[r] & \cI_{\WW_X/(X\times \check{\PP})}(h) \ar[r] \ar[d]  & \cO_{X \times \check{\PP}}(h) \ar[d] \ar[r]& \cO_{\WW_X}(h) \ar[d] \ar[r] & 0\\
    & 0 & 0 & 0
    }
    \end{equation*}
    Since $X$ is non-degenerate and linearly normal, we have that $H^0(\cI_X(1))=H^1(\cI_X(1))=0$, therefore $\hat{\rhop}_*(\mathcal{K})=0$ and
    $R^1\hat{\rhop}_*(\mathcal{K}) \simeq \tilde{\rhop}_*(\cI_{\WW_X/\II}(h))$.
    Combining it with the pushforward by $\hat{\rhop}$ of the left vertical sequence in the previous diagram concludes the proof.
\end{proof}

The next result describes more in detail the case where $X$ is 0-dimensional, for which $D = X^\vee$ is a hyperplane arrangement. In particular, we recover a result due to Faenzi and Vallès (see \cite{FaenziValles}).

\begin{corollary}
        Let $X$ be as in Theorem \ref{thm-pushforwardlog}. Then if $\dim X =0$, we have that 
        $$
        \TD(-1) \simeq \tilde{\rhop}_*\tilde{\pip}^* \cI_X(1).
        $$
\end{corollary}
\begin{proof}
    This result follows directly from the fact that, being $\dim X = 0$, we have $\WW_X = (X \times \check{\PP}) \cap \II$ and therefore
    $
    \tilde{\pip}^* \cI_X(1) \simeq \cI_{\WW_X / \II}(h)
    $.
\end{proof}

\begin{remark}
At the level of the Jacobian ideal sheaf, we have
$\rhop_*(\pip^*(\cO_X(1))) \simeq \cJ_D(d-1)$ if and only if  $R^1\tilde{\rhop}_*(\cI_{\WW_X/\II}(h))=0$. Indeed, this is tantamount to surjectiveness of $\tilde{\rhop}_*(\cO_\II (h)) \rightarrow \rhop_*\pip^*\cO_X(1)$ and thus to the vanishing of the higher direct image sheaf by Diagram (\ref{diag-logaritmictang}).
\end{remark}

The remaining of this section will be devoted to the discussion on the previous vanishing of the higher direct image. To do so, let us first recall the notions that will be necessary to apply the \textit{Theorem of formal functions} (for more details, see \cite[III.11]{Hartshorne}).\\
Considering a morphism of schemes $f:X \rightarrow S$ and a point $s \in S$, we would like to describe the fiber product 
$$
X_k = X \times_S \Spec\left(\cO_{S,s} / \mathfrak{m}_s^k\right).
$$
Taking an affine covering, which allows to restrict our description to such case, denote by $\phi:A\rightarrow B$ the associate ring map and by $\mathfrak{q}$ the prime ideal of $A$ corresponding to the point $s$. Therefore $X_k$ will be given by the spectrum of
$$
B \otimes_A A_{\mathfrak{q}} / \mathfrak{q}^n A_{\mathfrak{q}} \simeq B \otimes_A \left(A / \mathfrak{q}^n A \right)_{\mathfrak{q}} \simeq \left( B / \phi(\mathfrak{q})^n B \right)_{\phi(\mathfrak{q})}.
$$
Moreover if $s$ is a closed point, which means $\mathfrak{q}$ to be maximal, we do not need to localize.\\ 
Applied to our case, consider the projection $\tilde{\rhop} : \II \rightarrow \check{\PP}$ and a point of the dual projective space $y \in \check{\PP}$. We will denote both by $H_y$ the correspondent hyperplane in $\PP$ and also the fiber $\tilde{\rhop}^{-1}(y)$, seen as a subscheme of the incidence variety. Using the introduced notation on the thickened fibers, we have that $\II_1$ is simply the fiber $H_y$.\\ 
Furthermore, the previous considerations on the local description imply the following short exact sequence
$$
0 \rightarrow \cI_{H_y}^k \rightarrow \cO_{\II} \rightarrow \cO_{\II_k} \rightarrow 0,
$$
which means that the \textit{thickened fiber} $\II_k$ is defined by the $k$-th power of the ideal defining $H_y$ as a subvariety of $\II$. This implies the subsequent short exact sequence
\begin{equation}\label{thick-fiber}
0 \rightarrow \frac{\cI^{k-1}_{H_y}}{\cI^{k}_{H_y}} \rightarrow \cO_{\II_k} \rightarrow \cO_{\II_{k-1}} \rightarrow 0.
\end{equation}
Being the ideal of $H_y$, in the coordinate ring of the incidence variety, defined by a regular sequence, we know by \cite[II-8.21A(e)]{Hartshorne} that 
$$
{\cI^{k-1}_{H_y}}/{\cI^{k}_{H_y}} \simeq S^{k-1} \left(\cI_{H_y}/\cI^{2}_{H_y}\right),
$$
where $S^{k-1}$ denotes the $(k-1)$-st symmetric power.
Finally, notice that $H_y$ is a smooth complete intersection of $n$ elements in the linear system $|\check{h}|$. This implies that $\cI_{H_y}/\cI^{2}_{H_y}$ gives the conormal bundle associated to $H_y$ and, specifically, 
$$
\frac{\cI_{H_y}}{\cI^{2}_{H_y}} \simeq (\cO_{\II}(\check{h})^n)_{H_y} \simeq \cO_{H_y}^n.
$$
Substituting in (\ref{thick-fiber}), we obtain that having the vanishing
$$
H^1\left(H_y, (\cI_{\WW_X/\II}(h))_{H_y}\right) = 0
$$
implies all the subsequent ones, on the thickened fibers,
$$
H^1\left(\II_k, (\cI_{\WW_X/\II}(h))_{\II_k}\right) = 0.
$$
We are therefore in the position to state the following result.
\begin{lemma}\label{lemma-vanishingFormal}
If $H^1\left(H_y, (\cI_{\WW_X/\II}(h))_{H_y}\right) = 0$ holds for every point $y \in \check{\PP}$, then we have the isomorphism 
$$
\rhop_*(\pip^*(\cO_X(1))) \simeq \cJ_D(d-1).
$$
\end{lemma}
\begin{proof}
By the Theorem of formal functions, all the vanishings induced by $H^1\left(H_y, (\cI_{\WW_X/\II}(h))_{H_y}\right) = 0$ imply that $R^1\tilde{\rhop}_*(\cI_{\WW_X/\II}(1,0)) = 0$. By Diagram \ref{diag-logaritmictang}, this is equivalent to the required isomorphism.
\end{proof}

From a first glance, the cohomological vanishing of the previous lemma seems very restrictive. That is the reason why we will now study the example of plane curves. Nevertheless, such a restrictiveness highlights the choice of considering \textit{adjoint varieties}, for which we will prove the isomorphism always to be true.

\begin{example}
Consider a plane curve $X \in |\cO_{\PP^2}(d)|$ in the projective plane. Let us fix a point $y \in \check{\PP}^2$ and denote, as before, the corresponding line in $\PP^2$ by $H_y$.\\
First of all, let us consider the restriction to $H_y$ of the short exact sequence defining $\WW_X \subset \II$, that gives us
$$
\xymatrix@-2ex{
0 \ar[r] & \mathcal{T}or^1 (\cO_{\WW_X}(h), \cO_{H_y}) \ar[r] & \left(\cI_{\WW_X/\II}(h)\right)_{|H_y} \ar[rr] \ar[rd] & & \cO_{H_y}(1) \ar[r] & \cO_{\WW_X \cap H_y}(1) \ar[r] & 0\\
& & & \cI_{\WW_X \cap H_y}(1) \ar[ur] \ar[dr]\\
& & 0 \ar[ur] & & 0
}
$$
We can now divide our studying in the following cases:
\begin{itemize}
    \item $H_y$ is an irreducible component of $\WW_X$, which implies that the line $\tilde{\pip}(H_y)$ is also an irreducible component of $X$.
    \item $H_y$ is not an irreducible component of $\WW_X$, and neither its projection on $X$ is a component of $X$.
\end{itemize}
In the latter case, the support of $\mathcal{T}or^1 (\cO_{\WW_X}, \cO_{H_y})$ is at most 0-dimensional, which implies the vanishing of its first and second cohomology group. Therefore, $h^0\left(\cO_{\WW_X \cap H_y}(1)\right)\geq 3$ is equivalent to $H^1\left(\left(\cI_{\WW_X/\II}(h)\right)_{|H_y}\right)\neq 0$. Observing that $h^0\left(\cO_{\WW_X \cap H_y}(1)\right)$ is the number of points of $X$ tangent along $H_y$, we get that for any curve $X$ with (at least) a tritangent line, we cannot hope for the cohomological vanishings required in Lemma \ref{lemma-vanishingFormal}.

Regarding the first case, consider the following short exact sequence, obtained from the inclusions $H_y \subset \WW_X \subset \II$,
$$
0 \longrightarrow \cI_{\WW_X / \II} \longrightarrow \cI_{H_y / \II} \longrightarrow \cI_{H_y / \WW_X} \longrightarrow 0.
$$
Notice that $\cI_{H_y / \WW_X}$ is supported on the components of $\WW_X$ different from $H_y$ (denote $\WW_X = H_y \cup L$), therefore the restriction $\left(\cI_{H_y / \WW_X}\right)_{|H_y}$ will be given by the 0-dimensional scheme $Z$ defined by the intersection $L \cap H_y$. In particular, if $\length(Z) > 4$, then $H^1\left(\left(\cI_{\WW_X/\II}(h)\right)_{|H_y}\right)\neq 0$. 
\end{example}

\section{The geometric method and the Jacobian ideal}\label{sec-Weymanmethod}
    
Here we describe our use of the \textit{geometric method}, in the terminology of \cite{Weyman}, to compute resolutions of sheaves of logarithmic differentials.
This approach has been used extensively in \cite{gelfand-kapranov-zelevinsky} in the framework of projective duality, where the authors use the so-called Cayley method in order to obtain a complex resolving the discriminant, called the discriminant complex. It was already pointed out in \cite{tevelev} that the geometric method could be used to obtain this discriminant complex, but it seems that no application to logarithmic sheaves has been proposed so far. Describing such an application is the main goal of this section.

\subsection{The geometric method and normalisation}

Let $X\subset \PP=\PP(V)$ be a smooth connected projective variety of dimension $N$ with $V$ a vector space of dimension $\ell$.

Consider the affine conormal bundle sequence:
\[
0 \to \cN_{X/\PP}^\vee(1) \to V \otimes \cO_X \to \hat \Omega_X(1)  \to 0.
\]
The total space $\rN_{X/\PP}^\vee(1)=\Tot(\cN_{X/\PP}^\vee(1))$ is a subvariety of the trivial bundle $\Tot(V\otimes \cO_X) \simeq X\times \Vv$, of codimension equal to $N+1$, defined in the fibres of the projection $\overline{\rho} : X\times \Vv \to \Vv$ by the equations in $\hat \Omega_X(1) \simeq \left(\hat \cT_X(-1)\right)^\vee$. 
Then, there exists an exact \emph{Koszul complex} of the form:
\begin{align*}
     0 \to \wedge^{N+1} \overline{\pip}^*(\hat \cT_X(-1)) &\to 
 \wedge^{N} \overline{\pip}^*(\hat \cT_X(-1))  \to \cdots \\
  \to \cdots \to & \overline{\pip}^*(\hat \cT_X(-1))  \to \cO_{X\times \Vv} \to \cO_{\rN_{X/\PP}^\vee(1)} \to 0, 
\end{align*}
where $\overline{\pi} : X\times \Vv \to X$ is the natural projection onto the first component.
Recall that we write $\pi:\rN_{X/\PP}^\vee(1)\to X$ and $\rho:\rN_{X/\PP}^\vee(1)\to \Vv$ as the restrictions of $\pi$ and $\rho$ to $\rN_{X/\PP}^\vee(1)$ and, in this setting, the affine cone $\hat D \subset \Vv$ over the dual variety $D=X^\vee$ is the image of $\rho$.
Let $\cE$ be a vector bundle on $X$. Then the coherent sheaf 
$\rho_* (\pi^* (\cE))$ on $\Vv$ is supported on $\hat{D}$. By computing the pushforward $\overline{\rho}_*$ of the Koszul complex above, Weyman's theorem provides a complex 
\[
\bF^\cE_\bullet: \cdots \to \bF^\cE_{-1}\to \bF^\cE_0 \to \bF^\cE_{1}\to \cdots
\] 
whose terms, for $u \in \ZZ$, are of the form:
\begin{equation}
\label{eq_Weyman_res}
    \bF^\cE_{u}:=\bigoplus_{l-p=u} H^l(X,\wedge^p (\hat \cT_X(-1)) \otimes \cE)\otimes \bU(-p).
\end{equation}

\begin{remark}
If $G$ is an algebraic group acting on $X$ and both $\cE$ and $\cO_X(1)$ are $G$-linearised, then $G$ acts linearly on $\PP(V)$ and on $\PP(V^\vee)$. In this setup, the complex $\bF^\cE_\bullet$ will be $G$-equivariant.
\end{remark}

\begin{proposition} \label{normalisation}
Let $X\subset \PP$ be a smooth projective variety whose dual $D$ is a hypersurface.
\begin{enumerate}[label=\roman*)]
    \item \label{it's normal} If $H^l(X,\wedge^p \cT_X(-p))=0$ for $l-p>0$, then $\bF_{>0}^{\cO_X}=0$ and we have an exact complex: 
    \[ 
    0\to \bF_\bullet^{\cO_X} \to \rho_*\pi^*(\cO_X)\to 0.
    \]
    In addition, $\rho_*\pi^*(\cO_X)$ is the normalisation of $\hat D$, the affine cone over $D$.
    \item If moreover $H^l(X,\wedge^l \cT_X(-l))=0$ for $l>0$ then $\hat D$ is normal with rational singularities.
\end{enumerate}
Furthermore, if $X$ is Fano, namely $\omega_X^\vee$ is ample, then the vanishing required for \ref{it's normal} holds.
\end{proposition}

\begin{proof}
This result follows from a direct application of the geometric method. Notice that the morphism $\rho:\rN_{X/\PP}^\vee(1) \to \hat{D}$ is birational. Indeed,  $\hat{D}\subset \Vv$ is the affine cone over $D\subset \check{\PP}=\PP(V^\vee)$, which is a hypersurface, and thus $\rho$ is an isomorphism over the smooth locus of $\hat{D}$ by the biduality theorem. Then, to obtain the statements about $\bF^{\cO_X}_\bullet$ and the singularities of $\hat{D}$, apply \cite[Theorem 5.1.3]{Weyman}.

About the last assertion, assume that $X$ is Fano and note that, for all $p \ge 0$, we have:
\begin{equation} \label{iso wedges}
\wedge^p \cT_X(-p) \simeq \Omega^{N-p}_X(p) \otimes  \omega_X^\vee.
\end{equation}
Since we are in characteristic zero and since $\cO_X(p) \otimes \omega_X^\vee$ is ample, Akizuki-Nakano vanishing gives:
\[
H^l(\Omega^{N-p}_X \otimes \cO_X(p) \otimes  \omega_X^\vee) = 0, \qquad \text{for $l > p$},
\]
which, in view of \eqref{iso wedges}, is precisely \ref{it's normal}.
\end{proof}

\subsection{Resolution of the Jacobian ideal}

Let us come to our main application of the geometric method, namely the resolution of the Jacobian ideal.

\begin{theorem}
\label{thm_resolution_pushforward}
Let $X = \VV(F) \subset \PP$ be a smooth projective connected variety whose dual $D \subset \check \PP$ is a degree $d$ hypersurface. Let $J_D \subset \bU/(F)$ be the restricted Jacobian ideal of the affine cone $\hat D$ of $D$.
\begin{enumerate}[label=\roman*)]
    \item \label{thm : i} If $H^l(X,\wedge^p \cT_X(1-p))=0$ for $l-p>0$ then $\bF^{\cO_X(1)}_{>0}=0$ and we have an exact complex:
    \[
    0\to \bF^{\cO_X(1)}_\bullet \to \rho_*(\pi^*(\cO_X(1))) \to 0.
    \]
    \item \label{thm : ii}  
    If moreover $H^l(X,\wedge^l \cT_X(1-l))= 0$ for $l>0$ and $H^0(X,\cO_X(1))=\Vv$ then
    \[
    J_D(d-1)\simeq \rhop_*(\pip^*(\cO_X(1))).
    \]
\end{enumerate}
\end{theorem}

\begin{proof}
From the exact sequence \eqref{affine tangent}, for any integer $p$ with $0\leq p \leq N+1$, we get a short exact sequence 
\begin{equation} \label{wedge powers}
    0\to \wedge^{p-1}\cT_X \to \wedge^p\hat{\cT}_X \to \wedge^p \cT_X \to 0.
\end{equation}

Then, the vanishing appearing in \ref{thm : i} implies that $H^l(X,\wedge^p \hat{\cT}_X(1-p))=0$ for $l-p>0$. 
As in the proof of the previous proposition, the morphism $\rho:\rN_{X/\PP}^\vee(1) \to \hat{D}$ is birational. By applying \cite[Theorem 5.1.2]{Weyman}, we deduce that $\bF^{\cO_X(1)}_u=0$ for $u>0$ and hence the first statement above.

The hypothesis in the second statement implies moreover that $$\bF_0^{\cO_X(1)}=\tilde{\rhop}_*(\cO_\II(h))=\Vv \otimes \bU \simeq \Der_\bU.$$ 
By construction, the morphism 
$
\bF_0^{\cO_X(1)}\simeq \Der_\bU \to \rhop_*(\pip^*(\cO_X(1)))
$ 
factors through the surjective morphism $\Der_\bU \to J_D(d-1)$ and the inclusion $J_D(d-1) \to \rhop_*(\pip^*(\cO_X(1)))$. 
The commutativity argument required here is the same as in the proof of Theorem \ref{thm-pushforwardlog}.
The fact that $\bF_0^{\cO_X(1)} \to \rhop_*(\pip^*(\cO_X(1)))$ is surjective is a consequence of the first statement. This implies that the inclusion  $J_D(d-1) \to \rhop_*(\pip^*(\cO_X(1)))$ is also surjective, and thus an isomorphism.
\end{proof}

\begin{corollary}
In the hypothesis of items \ref{thm : i} and \ref{thm : ii} of Theorem \ref{thm_resolution_pushforward}, there is a free resolution 
\[0\to \cdots \to \bF_{-2}^{\cO_X(1)} \to \bF_{-1}^{\cO_X(1)} \to \Der_\bU(-\log(F)) \to 0.\]
\end{corollary}
\begin{proof}
Combine Theorem \ref{thm_resolution_pushforward} with the exact sequence \eqref{eq_affine_tangent_jac_modules}.
\end{proof}

\begin{remark}
    Working out the previous constructions in the projective setting, we can consider the maps $\pip:\PP(\rN_{X/\PP}(-1))\to X$ and $\rhop:\PP(\rN_{X/\PP}(-1))\to \PP(V^\vee)$ and a projective version of Weyman's complex, obtained by sheafification. We twist by one the resulting complex, so we write:
    \[
        \cF^\cE_{u}:=\bigoplus_{l-p=u} H^l(X,\wedge^p (\hat \cT_X(-1)) \otimes \cE)\otimes \cO_{\PP}(1-p).
    \]
    In the assumptions of items \ref{thm : i} and \ref{thm : ii} of Theorem \ref{thm_resolution_pushforward}, 
    we get a locally free resolution:
    \[
    0\to \cdots \to \cF_{-2}^{\cO_X(1)} \to \widetilde \cF_{-1}^{\cO_X(1)} \to \cT_\PP\langle D \rangle \to 0.
    \]
\end{remark}
    
In the following sections we will apply this result to the case of adjoint discriminants, namely, hypersurfaces obtained as projective dual varieties of adjoint varieties.

\section{Adjoint varieties, adjoint discriminants and their normalisations} \label{section : adjoint}

The goal of this section is to compute a resolution of the structure sheaf of the normalisation of an adjoint discriminant via the method developed in \S \ref{sec-Weymanmethod}.
This is mainly a warm-up for the next section, where we will compute a resolution of the Jacobian ideal of an adjoint discriminant. 
We start by recalling some basic features of adjoint varieties.
    
Let $G$ be a simple linear algebraic group over $\bk$, $V:=\fg$ its Lie algebra and $X$ the $G$-adjoint variety, i.e., the minimal $G$-orbit in $\PP(\fg)$. 
We will identify $\fg$ and $\fg^\vee$ via the Killing form, thus identifying canonically $V$ with $V^\vee$. Moreover we will denote by $\hat{\fg}$ the quasi-minuscule representation (i.e., the representation whose highest weight is the highest short root).
The dual variety $D=X^\vee \subset \PP(\fg)$ is a hypersurface, called the adjoint discriminant of $G$. It is the zero locus of a polynomial $\Delta$ of degree equal to the number of long roots of $G$ (see for instance \cite[Theorem 8.25]{tevelev}).

\subsection{Geometry of adjoint varieties}

Here we review some of the properties of adjoint varieties that will be useful to us.
An adjoint variety $X$ for a simple algebraic group $G$ is a $G$-homogeneous projective manifold and thus can be seen as a quotient $G/P$ for a certain parabolic subgroup $P\subset G$. 
Such a parabolic subgroup is associated to a subset $I_P$ of the simple roots $(\alpha_1,\ldots,\alpha_n)$. If $I_P=\{\alpha_{i_1},\ldots,\alpha_{i_s}\}$ we write $P=P_{i_1,\ldots,i_s}$. 

We will use the fundamental equivalence between representations of $P$ and $G$-homogeneous vector bundles on $X$.
According to it, irreducible homogeneous bundles on $X$ are in bijection with $P$-dominant weights, i.e., combinations $\varpi=\sum_i n_i\varpi_i$ with $n_i\in \ZZ$ for $i\in I_P$ and $n_i \ge 0$ for $i \not \in I_P$, where $\varpi_1,\ldots,\varpi_n$ are the fundamental weights with respect to the simple roots $\alpha_1,\dots,\alpha_n$ of $G$, which we will index according to Bourbaki's convention. If $\cE$ is a $G$-homogeneous vector bundle, then the associated representation of the parabolic group $P$, restricted to the semisimple part of $P$, gives rise to a homogeneous bundle which is a direct sum of irreducible bundles. We call this the \textit{semisimplification} of $\cE$ and we denote it by $\rs(\cE)$. Let $\cE_\varpi$ be the irreducible bundle corresponding to the weight $\varpi$ and, if $\varpi$ is $G$-dominant, write $V_\varpi$ for the $G$-representation of highest weight $\varpi$.\medskip

\noindent {\bf Bott-Borel-Weil Theorem :} Let $W$ be the Weyl group of $G$ and put $\varrho=\sum_i \alpha_i$ the half sum of all positive roots. Let $w\in W$ be the unique element such that $w(\varpi+\varrho)$ is $G$-dominant, i.e., $w(\varpi+\varrho)=\sum_i m_i\varpi_i$ with $m_i \ge 0$, and let us denote by $l(w)$ the length of $w$. The Bott-Borel-Weil theorem (\cite{bot}) asserts the following two statements. On the one hand, if there exists $j$ such that $m_j=0$ then $H^u(X,\cE_\varpi)=0$ for all $u$. On the other hand, if $m_i>0$ for all $i$ then $H^{l(w)}(X,\cE_\varpi)\cong V_{w(\varpi+\varrho)-\varrho}$ as $G$-representations and $H^u(X,\cE_\varpi)=0$ for all $u\neq l(w)$.

\subsubsection{Contact structure} 

\label{section:contact-structure} A key feature of  adjoint varieties is that they are contact manifolds (see \cite{beauville:fano-contact} and \cite{bucmor}). This means that there exists an exact sequence 
\begin{equation} \label{contact}
    0\to \cF \to \cT_X \xrightarrow{\theta} \cL \to 0,
\end{equation}
where $\cL$ is the line bundle defining the embedding $X\subset \PP(V)$ and $\cF$ is a vector bundle of rank $f=\dim(X)-1$ on $X$, equipped with a skew-symmetric self-duality induced by the differential of $\theta\in H^0(X,\Omega_X\otimes \cL)$ that, by abuse of notation, we will denote by $d\theta$:
\[
d\theta : \cF \xrightarrow{\sim} \cF^\vee\otimes \cL, \:\:\mbox{ with: } {}^{t}(d\theta)=-d\theta.
\]
In particular, $N=\dim(X)$ is odd, and we write $N=2e+1$. 

\begin{lemma}
The morphism $d\theta \wedge (\bullet) : \wedge^{p-2}\cF^\vee \to \wedge^p \cF^\vee\otimes \cL$ is an embedding for $p\leq e+1$.
\end{lemma}

\begin{proof}
This follows directly from the non-degeneracy of $d\theta$ as an element in $H^0(X,\wedge^2 \cF^\vee\otimes \cL)$.
\end{proof}

From the contact structure we get the following exact sequence: 
\begin{equation} \label{eq_omegap_F}
0\to \wedge^{p-1}\cF^\vee \to \Omega^p_X\otimes \cL \to \wedge^p \cF^\vee\otimes \cL \to 0.
\end{equation}

In all types except $A_n$, $\cF$ is a homogeneous irreducible bundle, thus the semisimplification of $\hat{\Omega}^p_X \otimes \cL$ is given by $$ \rs(\hat{\Omega}^p_X \otimes \cL) = \wedge^{p-2}\cF^\vee \oplus \wedge^{p-1}\cF^\vee \oplus \wedge^{p-1}\cF^\vee\otimes \cL \oplus \wedge^{p}\cF^\vee\otimes \cL.$$
In type $A_n$ the terms of the above decomposition are not irreducible (as we will soon see).

Let us give a brief introduction to adjoint varieties of classical groups.

\subsubsection{Type $A_n$. The point-hyperplane incidence variety} 

We have $G \simeq \SL_{n+1}$ and $\fg$ is the algebra of traceless matrices $\fsl_{n+1}$. 
The adjoint variety $X$ is $\PP(T_{\PP^n})$ and $I_P=\{\alpha_1,\alpha_n\}$, so $P=P_{1,n}$.
The variety $X$ is identified with the point-hyperplane incidence variety, namely a smooth hyperplane section of $\PP^n \times \check \PP^n$, and its Picard group is generated by two line bundles $\cO_X(1,0) \simeq \cE_{\varpi_1}$ and $\cO_X(0,1) \simeq \cE_{\varpi_n}$ obtained by pull-back from the two projections onto 
$\PP^n$ and $\check \PP^n$. 
The line bundle $\cL$ in \eqref{contact} is $\cL \simeq \cO_X(1,1)$. Finally we have $\rs(\cF)=\cE_{-\varpi_1+\varpi_2+\varpi_n}\oplus \cE_{\varpi_1+\varpi_{n-1}-\varpi_n}$.

\subsubsection{Type $B_n$ and $D_n$. The orthogonal Grassmannian of lines in odd and even dimension.} Let us consider the group $\SO(m)$ acting on a vector space $\bk^{m}$ and preserving a non-degenerate symmetric bilinear form on it. If $m$ is even, say $m=2n$, the Dynkin diagram of the Lie algebra $\fso_{m}$ is of type $D_n$. For odd $m$, with $m=2n+1$, the corresponding Dynkin diagram is of type $B_n$. 
The adjoint variety $X$ is the orthogonal Grassmannian of planes $OG(2,m)$ parametrizing linear subspaces $\bk^2 \subset \bk^{m}$ which are isotropic with respect to the symmetric form. The variety $X$ is a subvariety of the Grassmannian $G(2,m)$ and the line bundle $\cL$, defined by the contact structure, is the restriction of the Pl\"ucker line bundle on $G(2,m)$. This line bundle gives the embedding $X\subset \PP(\fso_{m})=\PP(\wedge^2 \bk^{m})$. 

Let $\cU$ denote the rank-2 tautological bundle on $OG(2,m)$, which is the restriction of the tautological bundle on $G(2,m)$. Furthermore, define $\cU^\perp/\cU$ as a rank $m-4$ bundle,  where its fibers are given by the quotient of the orthogonal space to the fiber of $\cU$ by the fiber of $\cU$ itself. We then have the bundle $\cF=\cU^\vee \otimes \cU^\perp/\cU$, and in terms of weights, $\cF=\cE_{\varpi_1-\varpi_2+\varpi_3}$. 

In type $B_n$ the quasi-minuscule representation is  $\hat{\fg}=\bk^{2n+1}=V_{\varpi_1}$.

\subsubsection{Type $C_n$. The Veronese embedding.}

Fix a non-degenerate skew-symmetric bilinear form $\omega$ on $\bk^{2n}$ and consider the group $G=\Sp(2n)$ of linear automorphisms of $\bk^{2n}$ preserving $\omega$. It is classical to use $\omega$ to identify the vectors of $\bk^{2n}$ with 1-forms on its dual, and endomorphisms of $\bk^{2n}$ with 2-tensors of $\bk^{2n}\otimes \bk^{2n}$. Under this identification, the elements of the Lie algebra $\fsp_{2n}$ correspond precisely to symmetric tensors. Hence, we may look at the adjoint representation of $\fsp_{2n}$ as $S^2 \bk^{2n}$. One defines the space of primitive 2-forms with respect to $\omega$, denoted by
$\wedge^{\langle 2 \rangle }\bk^{2n}$, as the kernel of the contraction map induced by $\omega$, i.e.,
$$
\begin{array}{rccc}
& \wedge^2 \bk^{2n} & \longrightarrow & \bk \\
& u \wedge v & \mapsto & \omega(u,v).
\end{array}
$$
It turns out that $\wedge^{\langle 2 \rangle }\bk^{2n}$ is an irreducible $\Sp(2n)$-representation of highest weight $\varpi_2$. This is the representation that we call quasi-minuscule and that we denote by $\hat \fg$, when $\fg=\fsp_{2n}$. Summing up:
\[
\hat \fg = V_{\varpi_2} \simeq \wedge^{\langle 2 \rangle }\bk^{2n}.
\]

The adjoint variety is $X=v_2(\PP(\bk^{2n}))$, the second Veronese embedding of the projective space.
In this case, the line bundle of the contact structure is $\cL=\cO_{\PP^{2n-1}}(2)$ and the dual sequence of \eqref{contact} reads
\[
0 \to \cO_{\PP^{2n-1}}(-2) \to  \Omega_{\PP^{2n-1}} \to \cF^\vee \to 0.
\]
Here, $\cF^\vee \otimes \cO_{\PP^{2n-1}}(1)$ is a null-correlation bundle. 
\medskip

We postpone the description of the exceptional cases and their adjoint varieties to Section \ref{exc_types}. We just notice that the Betti numbers of these varieties can be found in \cite{percha}, while the exponents of simple linear algebraic groups can be found in \cite{Liebook}. We end this section with a useful lemma.
\begin{lemma} \label{tangent & cotangent}
\label{lem_iso_xi}
Let $X$ be an adjoint variety. Then we have an isomorphism $\hat{\cT}_X\otimes \cL^\vee\cong \hat{\Omega}_X$.
\end{lemma}

\begin{proof}
Recall from \S \ref{section:contact-structure} that we denoted $\dim X = N = 2e+1$. First notice that $\cL$ is the restriction to $X$ of the hyperplane class of $\PP(\fg)$. Since $\cL$ is ample, the Kodaira vanishing theorem implies that $H^q(\cL^\vee)=0$ for all $q<N$.  By Serre duality,
\[
H^{N}(X,\mathcal L^{\vee}) \simeq H^0(X,\mathcal L \otimes \omega_X)^{\vee}.
\]
Since $c_1(\mathcal L \otimes \omega_X) = -e c_1(\mathcal L)$ and $\mathcal L$ is ample, we obtain $H^{2e+1}(X,\mathcal L^{\vee}) = 0$. Observe that the latter vanishing holds, more in general, for any Fano variety of index $\iota>1$.
Write the dual of the exact sequence \eqref{contact} defining the contact structure as
\begin{equation} \label{contact-structure}
    0\to \cL^\vee \to \Omega_X \to \cF^\vee \to 0.
\end{equation} 
This yields, for all $u$, 
$$H^u(X,\cF\otimes \cL^\vee)=H^u(X,\cF^\vee)=H^u(X,\Omega_X).$$
We have another natural exact sequence, namely the sequence defining the affine tangent bundle of $X$, which reads
\begin{equation} \label{affine-tangent-bundle}
0\to \cL^\vee \to \hat{\cT}_X\otimes \cL^\vee \to \cT_X\otimes \cL^\vee \to 0    
\end{equation}
Both sequence \eqref{contact-structure} and \eqref{affine-tangent-bundle}
are induced by the element in $H^1(X,\cF^\vee)=H^1(X,\Omega_X)$ corresponding to the hyperplane class. Then the inclusion $\cF^\vee \cong \cF\otimes \cL^\vee \to \cT_X\otimes \cL^\vee$ induces an inclusion $\Omega_X \to \hat{\cT}_X\otimes \cL^\vee $. We get the following commutative diagram:
\begin{equation*}
    \xymatrix@-1.5ex{
     & \cL^\vee \ar@{=}[r] \ar@{^(->}[d] & \cL^\vee  \ar@{^(->}[d] &    \\
    0 \ar[r] & \Omega_X \ar[r] \ar@{->>}[d] & \hat{\cT}_X\otimes \cL^\vee \ar[r] \ar@{->>}[d] & \cO_X \ar[r] \ar@{=}[d] & 0\\
    0 \ar[r] & \cF^\vee \ar[r] & \cT_X\otimes \cL^\vee \ar[r]& \cO_X \ar[r] & 0\\
    }
    \end{equation*}
The two bottom lines are thus induced, once again, by the same element in $H^1(X,\cF^\vee)=H^1(X,\Omega_X)$, and since the lower line is essentially the contact structure, which is induced by the hyperplane class, the upper line is also induced by the hyperplane class, and it is therefore the exact sequence defining the affine cotangent bundle. We obtain $\hat{\cT}_X\otimes \cL^\vee\cong \hat{\Omega}_X$. 
\end{proof}

\subsection{The normalisation of adjoint discriminants}
Let $D$ be the dual hypersurface variety of an adjoint variety $X$, with $\dim(X)=N=2e+1$, and denote by $\hat D^\nu \to \hat D$ the normalisation of the affine cone $\hat D$ over $D$. Thus, the coordinate ring $\bk[\hat D^\nu]$ is the integral closure of $\bU/(F)$, where $F$ is an equation of $D$.
Recall that $\bR:=\bk[\fg]^G$ is a polynomial algebra generated by $n$ polynomials $F_1,\ldots,F_n$ of degrees, respectively, $d_1,\ldots,d_n$, where $n$ is the rank of $G$. Reordering the polynomials if necessary, we can define the exponents $e_1 \leq \cdots \leq e_n$ of $G$ as $e_i:=d_i-1$ for $i=1,\ldots,n$. Finally, denote by $s$ the number of long simple roots. Our goal for this section is to prove the following result.

\begin{theorem}
\label{thm_res_structure}
A minimal graded free resolution of the coordinate ring $\bk[\hat D^\nu]$ is of the form :
\[
0\to \bigoplus_{i=1}^s \bU(-m+e_i-2) \to \bigoplus_{i=1}^s \bU(-e_i+1) \to \bk[\hat D^\nu] \to 0.
\]
\end{theorem}

\begin{corollary}
The structure sheaf of the normalisation $D^\nu$ of $D$ has the locally free resolution:
\[
0\to \bigoplus_{i=1}^s \cO_{\PP(\fg)}(-m+e_i-2) \to \bigoplus_{i=1}^s \cO_{\PP(\fg)}(-e_i+1) \to \cO_{D^\nu} \to 0.
\]
\end{corollary}

We start with the following lemma, which for classical groups is just \cite[Exercises 9.3, 9.4, 9.5, 9.6]{Weyman}.

\begin{lemma}
Let $X$ be an adjoint variety and set $N=\dim(X)$. There exists a $G$-equivariant resolution
\[
0\to \bigoplus_{p=0}^N H^p(X, \Omega^p_X)\otimes \bU(-p-1) \to
\bigoplus_{p=0}^N H^p(X, \Omega^p_X)\otimes \bU(-p)  \to  \bk[\hat D^\nu]  \to 0.
\]
\end{lemma}

\begin{proof}
We apply Theorem \ref{thm_resolution_pushforward} to get a resolution of the normalisation of the affine cone $\hat{D}$ of $D$. The terms of the resolution are given by $\bF_u^{\cO_X}$ for all $u \le 0$, and these are computed from the cohomology of $\wedge^p(\hat{\cT}_X\otimes \cL^\vee)$. This computation yields the statement by using Lemma \ref{tangent & cotangent} and the dual of \eqref{wedge powers}.
\end{proof}

\begin{proof}[Proof of Theorem \ref{thm_res_structure}]
For all $i \ge 0$, consider the map given by multiplication by the hyperplane class: 
$$H^{i-1}(X, \Omega^{i-1}_X)\to H^i(X, \Omega^i_X),$$
and denote by $K^{i-1}$ and $C^i$, respectively, the kernel and cokernel of such map. The non-trivial extension 
$$0 \to \Omega_X \to \hat \Omega_X \to \cO_X \to 0$$ 
induces a map in cohomology $H^0(X, \cO_X) \to H^1(X, \Omega^1_X)$ whose image is the hyperplane class corresponding to the embedding $X\subset \PP(V)$. Thus, taking duals in \eqref{wedge powers} we find, for any $p \ge 0$, an exact sequence
\begin{equation}
\label{wedge power omega}    
0 \to \Omega^p_X \to \hat \Omega^p_X \to \Omega^{p-1}_X \to 0,
\end{equation}
whose induced maps in cohomology $H^{p-1}(X, \Omega^{p-1}_X) \to H^{p}(X, \Omega^{p}_X)$ are just given by the multiplication by the hyperplane class. 
By the Lefschetz hyperplane theorem, we deduce that, for all $p \in \ZZ$, the maps 
$$H^{p-1}(X, \Omega^{p-1}_X)\otimes \bU(-p) \to H^{p}(X, \Omega^{p}_X) \otimes \bU(-p)$$ appearing in $\bF_\bullet^{\cO_X}$ have maximal rank. 
Therefore we obtain the following minimal resolution:
\[
0\to \bigoplus_{q=e+1}^N K^q\otimes \bU(-q-1) 
\to \bigoplus_{q=0}^e C^{q}\otimes \bU(-q) \to \bk[\hat D^\nu] \to 0.
\]
The result then follows by noticing, through an explicit comparison of Betti numbers for adjoint varieties and exponents for each simple linear algebraic group, that $C^i$ (and $K^{N-i}$) is a direct sum of $u_i$ trivial $G$-representations, where $u_i$ is the cardinality of $\{ j\mid \deg(f_j)=i+2\}$.
\end{proof}

\begin{corollary}
The adjoint discriminant $D$ is normal if and only if $G$ is of type $C_n$, or $B_2$, or $G_2$.
\end{corollary}

\begin{proof}
The adjoint varieties in type $C_n$, $B_2$ and $G_2$ are the only adjoint varieties whose rational cohomology is generated by the hyperplane class. This is reflected by the fact that for the corresponding linear algebraic groups $s=1$. By Theorem \ref{thm_res_structure} and Theorem \ref{normalisation} we deduce that the dual o  f these two adjoint varieties are the only ones which are normal.
\end{proof}

\section{Jacobian ideals of adjoint discriminants} 
\label{section: adjoint discriminants}

In this section we want to compute a $G$-equivariant free resolution of the restricted Jacobian ideal, and, as a consequence, of the module of logarithmic derivations, of the affine cone of the dual variety $D$ of an adjoint variety $X\subset \PP(\fg)$. In order to do so, we will use the method described in \S \ref{sec-Weymanmethod} to compute a minimal graded free resolution of $\rhop_*(\pip^*\cL)$.
Then we will identify $\rhop_*(\pip^*\cL)$ with the restricted Jacobian ideal of $D$. 
 For groups of simply laced type, this represents the first step to prove Theorem \ref{main simply laced}. For the non-simply laced case, the resolution we obtain is the one appearing in Theorem \ref{main not simply laced}.

In our setting, one may work indifferently with graded modules of finite type on $\bU=\bk[\fg]$ or coherent sheaves on $\PP(\fg)$.
We use the customary identification between $\fg$ and $\fg^\vee$ via the Killing form, hence, to ease the notation, most of the times we will just use $\fg$.

The following result ensures that we can apply Weyman's method in order to compute a locally free resolution of $\rhop_*(\pip^*(\cL^{\otimes i}))$, for any $i>0$ and for any adjoint variety $X$:

\begin{proposition}
Let $X$ be an adjoint variety and let $i>0$ be a positive integer. Then Weyman's complex $\bF^{\cL^{\otimes i}}_\bullet$ is a  graded free resolution of $\rhop_*(\pip^*(\cL^{\otimes i}))$.
\end{proposition}

\begin{proof}
By \cite[Theorem 5.1.2]{Weyman}, we need to check that $\bF^{\cL^{\otimes i}}_u=0$ for any $u>0$. Looking at the construction of Weyman's complex and using Lemma \ref{lem_iso_xi}, we see that

$$
\bF^{\cL^{\otimes i}}_u \simeq \bigoplus_{l-p=u} H^l(X,\wedge^p\hat{\Omega}_X \otimes \cL^{\otimes i})\otimes \bU(-p).
$$

For any $p \ge 0$ and $i \in \ZZ$, we tensor \eqref{wedge power omega} by $\cL^{\otimes i}$ to obtain an exact sequence
$$ 0\to \Omega_X^p\otimes \cL^{\otimes i} \to \hat{\Omega}^p_X \otimes \cL^{\otimes i} \to \Omega^{p-1}_X\otimes \cL^{\otimes i} \to 0 .$$

By \cite[Theorem 3.18]{brion_coh}, $H^l(X,\Omega^p_X\otimes \cL^{\otimes i})=0$ as soon as $i>0$ and $l>p$. From the above sequence, we deduce that $\bF_u^{\cL^{\otimes i}}=0$ for $u>0$. 
The result follows.
\end{proof}

Now we turn to the computation of Weyman's complex $\bF^{\cL}_\bullet$. We need to perform a certain number of computations on a case-by-case basis, although the final result admits a uniform formulation.
    
\subsection{Type $C_n$: Warming up}

In this case $X=v_2(\PP^{2n-1})\subset \PP(S^2 \bk^{2n})=\PP$ is the second Veronese embedding of the projective space and $\cL=\cO_{\PP^{2n-1}}(2)$. Recall that $\hat{\fsp}_{2n}=V_{\varpi_2}$.

\begin{proposition}
Let $X=v_2(\PP^{2n-1})$ be the adjoint variety of type $C_n$ and $D= \VV(\Delta) = 
X^\vee$ the adjoint discriminant. 
We have an isomorphism $\rhop_*(\pip^*(\cL)) \simeq J_D(d-1)$ and an equivariant  minimal graded free resolution:
$$ 
0\to \wedge^2 \bk^{2n}\otimes \bU(-2) \to (\bk^{2n}\otimes \bk^{2n}) \otimes \bU(-1) \to S^2 \bk^{2n} \otimes \bU \to J_D(d-1)\to 0.
$$
Equivalently, we have an equivariant graded free resolution:
\[
    0\to \left(\bk \oplus \hat{\fsp}_{2n}\right) \otimes \bU(-2) \to (\bk \oplus \hat{\fsp}_{2n} \oplus \fsp_{2n}) \otimes \bU(-1) \to 
    \Der_\bU(-\log(\Delta)) \to 0.
\]
\end{proposition}

\begin{proof}
Let us denote by $\cQ$ the tautological quotient bundle of $X$, namely $\cQ = \cT_{\PP(\fsp_{2n})}\otimes \cO_{\PP(\fsp_{2n})}(-1)$. 
Then, coming from the splitting of the Koszul complex, we have short exact sequences
$$ 0\to \wedge^p\cQ^\vee(2-p) \to \hat{\Omega}^p_X \otimes \cL \to \wedge^{p-1}\cQ^\vee(3-p)\to 0 \mbox{ for all } p\ge 0. $$

The only non-vanishing cohomology groups of $\wedge^p\cQ^\vee(2-p)$ are given by $H^0(\PP(\fsp_{2n}), \cO_{\PP(\fsp_{2n})}(2))\cong S^2 \bk^{2n}\cong V $, and $H^0(\PP(\fsp_{2n}),\cQ^\vee(1))\cong \wedge^2 \bk^{2n}$. Thus the only non-vanishing cohomology groups of $\hat{\Omega}^p_X \otimes \cL$ are:
\begin{align*}
    &H^0(\PP(\fsp_{2n}),\cL)\cong S^2\bk^{2n}, \\
    &H^0(\PP(\fsp_{2n}),\hat{\Omega}^1_X \otimes \cL)\cong \bk^{2n}\otimes \bk^{2n}, \\
    &H^0(\PP(\fsp_{2n}),\hat{\Omega}^2_X \otimes \cL)\cong \wedge^2 \bk^{2n}.
\end{align*}
The result now follows by noticing that  $S^2\bk^{2n}\otimes\cO_{\PP(\fsp_{2n})}\cong \Der_\bU$. 
\end{proof} 

Notice that the previous result was already obtained in \cite{faenzi-marchesi} using a different technique, namely as the resolution of the tangent logarithmic sheaf of the discriminant of quadratic forms.
    
\subsection{Type $A_n$}

In this case $X$ is the flag variety $F(1,n,n+1)$ parametrizing flags $\bk\subset \bk^n$ inside a fixed $n+1$ dimensional vector space $A$. It can also be seen as the projectivization of the cotangent bundle of the projective space $\PP^n$. The group $G$ is in this case $\SL(n+1)$ and the associated Weyl group is the group of permutations of $n+1$ elements. The bundle $\cL$ is the ample line bundle $\cO_X(1,1)$ defining the embedding $X\subset \PP(\fsl_{n+1})$, which we also denote by $\cO_X(1)$ to ease the notation. Following \cite{Kuchle}, the weight associated to an irreducible homogeneous bundle over $X$ is given by a sequence of integers 
$$
\lambda:=[\lambda_1;\lambda_2,\ldots,\lambda_n;\lambda_{n+1}]=\sum_i (\lambda_i-\lambda_{i+1})\varpi_i
$$ 
such that $\lambda_2\geq \cdots \geq \lambda_n$. For instance, $\cO_X(1,0)$ is associated with the weight $[1;0,\dots,0;0]=[1;0^{n-1};0]$, while $\cO_X(1,1)$ is associated with  $[1;0\dots,0;-1]$. Notice that, in this notation, $\lambda$ and $\lambda+c:=[\lambda_1+c;\lambda_2+c,\dots,\lambda_n+c; \lambda_{n+1}+c]$, for any $c\in \ZZ$, are associated to the same irreducible homogeneous bundle.

\begin{remark}[Bott-Borel-Weil theorem for $\SL(n+1)$]
    The bundle $\cE_\lambda$ is globally generated if $\lambda_1\geq \lambda_2$ and $\lambda_n\geq \lambda_{n+1}$; in this case by the Bott-Borel-Weil theorem the space of sections $H^0(X,\cE_\lambda)$ is isomorphic to the $\SL(n+1)$-representation given by the Schur module $S_\lambda \bk^{n+1}$. More generally, let $\varrho:=[n+1,n,\dots,2,1]$ and consider the two following situations: either the integers in $\lambda+\varrho$ are all pairwise distinct, or two of them coincide. In the latter case, we have $H^i(X,\cE_\lambda)=0$ for all $i\geq 0$ by the Bott-Borel-Weil theorem. In the former case, denote by $w$ the permutation of $n+1$ elements such that $w(\lambda+\varrho)$ is a strictly decreasing sequence of integers. Then, again by the Bott-Borel-Weil theorem, $H^{l(w)}(X,\cE_\lambda)=S_{w(\lambda+\varrho)-\varrho}\bk^{n+1}$ and $H^i(X,\cE_\lambda)=0$ for all $i\neq l(w)$, where $l(w)$ is the \emph{length} of the permutation $w$ (i.e., the minimal number of simple permutations needed to obtain $w$).
\end{remark}

Notice that $\hat{\cT}_X^\vee\otimes \cL\cong \hat{\Omega}_X^\vee =(\fsl(A)\otimes \cO_X / \rN_{X/\PP}^\vee)^\vee$. Recall that, by the geometric method, the spectral sequence $H^{j}(X,\wedge^{i}\hat{\Omega}_X\otimes \cL)$ defines a locally free resolution $\{\oplus_{i} \bF^{\cL}_{j-i}(-i)\}_{j-i}$ of $\rhop_*\pip^*\cL$ as soon as $\bF^{\cL}_k=0$ for $k>0$. We will compute such a resolution by computing the cohomology of the graded pieces of $\hat{\Omega}_X$; by doing so, we will obtain a resolution which a priori is not minimal, but still resolves $\rhop_*\pip^*\cL$. The semisimplification of $\hat{\cT}_X\otimes \cL^\vee\cong \hat{\Omega}_X$ is given by 
$$
\rs(\hat{\Omega}_X)=\cO_X\oplus \cU_1 \otimes (A/\cU_n)^\vee \oplus \cU_1\otimes (\cU_n/\cU_1)^\vee \oplus (\cU_n/\cU_1)\otimes (A/\cU_n)^\vee,
$$
where $\cU_1$ and $\cU_n$ denote the tautological bundles of rank, respectively, $1$ and $n$ on $F(1,n,n+1)$.

In the notation of weights, it translates into:
$$
\rs(\hat{\Omega}_X)=[0;0,\ldots,0;0]+[-1;0,\ldots,0;1]+[-1;1,0,\ldots,0;0]+[0;0,\ldots,0,-1;1].
$$
The first two terms are line bundles, so they represent the easiest part to deal with. The last two terms have rank $n-1$ and, since we need to compute the cohomology of $\wedge^i\hat{\Omega}_X$, we give the formula to compute the exterior power of their sum: 
\begin{align*}
    \wedge^i(\rs(\hat{\Omega}_X))&=\wedge^i([-1;1,0,\ldots,0;0]+[0;0,\ldots,0,-1;1]) \oplus\\
& \oplus \wedge^{i-1}([-1;1,0,\ldots,0;0]+[0;0,\ldots,0,-1;1]) \oplus\\
&\oplus \wedge^{i-1}([-1;1,0,\ldots,0;0]+[0;0,\ldots,0,-1;1])\otimes [-1;0,\ldots,0;1] \oplus \\
&\oplus \wedge^{i-2}( [-1;1,0,\ldots,0;0]+[0;0,\ldots,0,-1;1])\otimes [-1;0,\ldots,0;1].
\end{align*}
In this decomposition the first factor is equal to $\wedge^i \cF^\vee$, the second factor is equal to $\wedge^{i-1}\cF^\vee$, the third factor is equal to $\wedge^{i-1}\cF^\vee\otimes \cL^\vee$ and the fourth factor is equal to $\wedge^{i-2}\cF^\vee\otimes \cL^\vee$, where $\cF$ is defined, as usual, by the contact structure. Tensoring by $\cL$, we obtain the semisimplification of $\wedge^i\hat{\Omega}_X\otimes \cL$.

\begin{lemma}
For $0\leq i\leq 2n-2$, $\wedge^i \cF^\vee\otimes \cL$ is identified with
$$
\bigoplus_{\substack{
            p,q\leq n-1,p+q=i ,\\
            \max(0,q-p)\leq j\leq\min(q,n-p-1)}}[-q+1;1^{j},0^{n-p+q-2j-1},(-1)^{p-q+j};p-1].
$$
For $1\leq i\leq 2n-1$, $\wedge^{i-1} \cF^\vee$ is identified with
$$ 
\bigoplus_{\substack{
            p,q\leq n-1,p+q=i-1 ,\\
            \max(0,q-p)\leq j\leq\min(q,n-p-1)}}[-q;1^j,0^{n-p+q-2j-1},(-1)^{p-q+j};p].
$$
\end{lemma}

\begin{proof}
In order to obtain the above descriptions, one needs to combine the formula for computing the exterior power of a direct sum of vector spaces 
\[\wedge^i (U\oplus W)=\bigoplus_{p+q=i}\wedge^p U \otimes \wedge^q W\]
with the Littlewood-Richardson (LR) rule for tensoring two representations
\[A_\mu\otimes A_\eta=\bigoplus_\lambda c_{\mu\eta}^\lambda A_\lambda.\]
Here $A_\mu$, $A_\eta$ and $A_\lambda$ are $\SL(A)$-representations of highest weight respectively $\mu$, $\eta$ and $\lambda$, and $c_{\mu\eta}^\lambda$ are the so-called Littlewood-Richardson coefficients, see for instance \cite{LRcoeff}. In our case, it suffices to consider the LR rule for the tensor product of exterior powers of the standard representations of $\SL(A)$, for which it assumes a particularly simple form:
\begin{align*}
    \wedge^q A  \otimes \wedge^p A^\vee &= S_{1^q}A\otimes S_{(-1)^p}A= \\ = & \bigoplus_{\max(0,q-p)\leq j\leq\min(q,\dim(A)-p)}  S_{1^j,0^{\dim(A)-p+q-2j},(-1)^{p-q+j}}A.
\end{align*}

One obtains
\begin{align*}
\wedge^i & ([-1;1,0,\ldots,0;0]+[0;0,\ldots,0,-1;1])= \\
&=\bigoplus_{\substack{
            p,q\leq n-1,p+q=i ,\\
            \max(0,q-p)\leq j\leq\min(q,n-p-1)}}[-q;1^j,0^{n-p+q-2j-1},(-1)^{p-q+j};p],
\end{align*}
and the result follows.
\end{proof}

\begin{proposition}
\label{lem_coh_An}
All cohomology groups of $\wedge^i \cF^\vee$ and $\wedge^i \cF^\vee\otimes \cL$ vanish except for the following ones:
\begin{itemize}
    \item $H^i (X,\wedge^i \cF^\vee)\cong H^i(X, \Omega_X^i) $ for $i\leq e$, where the isomorphism is induced by $\Omega^i_X\twoheadrightarrow \wedge^i \cF^\vee $;
    \item $ H^0(X,\wedge^0\cF^\vee\otimes \cL)\cong \fsl_{n+1}$;
    \item $ H^{i-2} (X,\wedge^i \cF^\vee\otimes \cL)\cong H^{i-2}(X, \Omega_X^{i-2})$ for $2\leq i\leq e-1$. Moreover, the terms $ H^{i-2}(X, \Omega_X^{i-2}) $ are the images of the maps induced in cohomology by the embeddings $d\theta\wedge (\bullet):\wedge^{i-2}\cF^\vee \to \wedge^i \cF^\vee\otimes \cL$.
\end{itemize}
\end{proposition}

\begin{proof}
Let us compute the cohomology of the bundles by using their weight decomposition.
\begin{itemize}
    \item[$\wedge^{i-1}\cF^\vee$:] Consider the factor in $\wedge^{i-1}\cF^\vee$ given by $[-q;1^j,0^{n-p+q-2j-1},(-1)^{p-q+j};p]$. It admits non-vanishing cohomology if and only if no integer in the sequence 
    $$
    [-q;1^j,0^{n-p+q-2j-1},(-1)^{p-q+j};p]+[n+1;n,\dots,2;1]
    $$
    is repeated. Let us look for integers $p,q,i,j$ for which the cohomology does not vanish. Either $q=j$ or $q= n-p+q$; indeed, if $q\geq n-p+q-1$, then $q\geq n+1$, which is impossible. Similarly, either $p=p-q+j$ or $p= n-j$. However, $q=n-p+q$ is impossible because $p\leq n-1$, which also excludes $p=n-j$. Therefore we only left with the possible case $q=j$, which implies $p=p-q+j$. For these terms, the cohomology is concentrated in degree $j+p-q+j=i-1$. Thus we obtain $H^{i-1}(\wedge^{i-1}\cF^\vee)=\bk^{i}$ if and only if $1\leq i\leq n$, since $q\leq n-p-1$ and $q$ runs from $0$ to $i-1$.
    \item[$\wedge^i \cF^\vee\otimes \cL$:] Consider the factor in $\wedge^i \cF^\vee\otimes \cL$ given by $[-q+1;1^j,0^{n-p+q-2j-1},(-1)^{p-q+j};p-1]$. It admits non-vanishing cohomology if and only if no integer in the sequence 
    $$
    [-q+1;1^j,0^{n-p+q-2j-1},(-1)^{p-q+j};p-1]+[n+1;n,\dots,2;1]
    $$
    is repeated. Let us look for integers $p,q,i,j$ for which the cohomology does not vanish. One possibility is of course given by $p=q=i=j=0$, for which $H^0(\cL)=\fsl(A)$. 
    If $q\neq 0$ either $q=j+1$ or $q\geq n-p+q+1$; however, the latter is impossible since $p\leq n-1$. Similarly, if $p\neq 0$ either $p=p-q+j+1$ or $p\geq n-j+1$; however, the latter is impossible since $j\leq n-p-1$. Thus only we are left with the possibility $q=j+1$, which implies $p=p-q+j+1$. For these terms, the cohomology is concentrated in degree $j+p-q+j=i-2$. Thus we obtain $H^{i-2}(\wedge^i \cF^\vee\otimes \cL)=\bk^{i-1}$ if and only if $2\leq i\leq n$, since $q\geq 1$, $p\geq 1$, $q-1=j\leq n-p-1$ and $q$ runs from $1$ to $i-1$.
\end{itemize}
By Serre duality, we have isomorphisms
\[H^{i}(X,\wedge^i \cF^\vee\otimes \cL^\vee)^\vee\cong H^{N+1-i}(X,\wedge^{i}\cF\otimes \cL\otimes \omega_X)\cong H^{N+1-i}(X,\wedge^{N-i}\cF^\vee).\] 
Since $\rs(\Omega^i_X)=\wedge^{i-1}\cF^\vee\otimes \cL^\vee\oplus \wedge^i \cF^\vee$ and $\wedge^{i-1}\cF^\vee\otimes \cL^\vee$ is acyclic for $i\leq e$, we deduce that the cohomology groups of $\wedge^i \cF^\vee$, for each $i\leq e$, are induced by the respective surjection $\Omega^i_X\to \wedge^i \cF^\vee $.

In order to show that the terms $H^{i-2}(\Omega_X^{i-2})$ in the cohomology of $\wedge^i \cF^\vee\otimes \cL$ are induced by $d\theta\wedge (\bullet)$ notice that, by the explicit computations above, each irreducible bundle in $\rs(\wedge^{i-2}\cF^\vee)$ appears only once in the decomposition of the semisimplification. Moreover, its cohomology (when it does not vanish) appears both in $H^{i-2}(X,\wedge^{i-2}\cF^\vee)$ and $H^{i-2}(X,\wedge^{i}\cF^\vee\otimes \cL)$. The claim follows.
\end{proof}

\subsection{Type $B_n$ and $D_n$: orthogonal Grassmannians of planes}
A more comprehensive description of these cases can be found in \cite{ben}, from which we follow the notation.

Let us consider the group $\SO(m)$, which is of type $B_n$ when $m=2n+1$ and of type $D_n$ when $m=2n$. From now on, to uniformize notation, we will define $h=1/2$ (respectively $h=0$) in type $B_n$ (respectively $D_n$), so that $m=2(n+h)$. The adjoint variety $X$ is the orthogonal Grassmannian of planes $OG(2,m)$ parametrizing isotropic subspaces $\bk^2 \subset \bk^m$ inside a fixed $m$ dimensional vector space endowed with a symmetric $2$-form. The bundle $\cL$ is the ample line bundle defining the embedding $X\subset \PP(\fso_{m})=\PP(\wedge^2 \bk^m)$, and we will again denote $\cL$ by $\cO_X(1)$ to ease the notation.

The weight associated to an irreducible homogeneous bundle over $X$ is given by a sequence, made only of integers or half-integers, 
$$
\lambda:=[\lambda_1,\lambda_2;\lambda_3,\ldots,\lambda_{n}]
$$
such that $\lambda_1\geq \lambda_2$, $\lambda_3\geq \cdots \geq \lambda_n$ and: $\lambda_n\geq 0$ in type $B_n$ and $\lambda_{n-1}+\lambda_n\geq 0$ in type $D_n$. In terms of fundamental weights $\lambda=\sum_{i\leq n-1} (\lambda_i-\lambda_{i-1})\varpi_i+ 2\lambda_n\varpi_n$ in type $B_n$ and $\lambda=\sum_{i\leq n-1} (\lambda_i-\lambda_{i-1})\varpi_i+ (\lambda_{n-1}+\lambda_n)\varpi_n$ in type $D_n$.

\begin{remark}[Bott-Borel-Weil theorem for $\SO(m)$]
The bundle $\cE_\lambda$ is globally generated if $\lambda_2\geq \lambda_3$; in this case, by the Bott-Borel-Weil theorem, the space of sections $H^0(X,\cE_\lambda)$ is isomorphic to the $\SO(m)$-representation $V_\lambda$ with highest weight $\lambda$. More generally, let $\varrho:=[n-1+h,n-2+h,\dots,1+h,h]$ and consider the two following situations: either the integers (or half-integers) in $\lambda+\varrho$ union $-\lambda-\varrho$ are all pairwise distinct except for $0$ which can appear twice; or there are two non-zero integers that coincide. In the latter case, we have $H^i(X,\cE_\lambda)=0$ for all $i\geq 0$ by the Bott-Borel-Weil theorem. In the former case, consider the Weyl group $W$ of $\SO(m)$ which is a semidirect product of the permutations of $n$ elements of $\lambda$ and $\ZZ_2$. The group $\ZZ_2$ is generated by the reflection $\tau$ that exchanges $\lambda_n$ into $-\lambda_{n}$ in type $B_n$ and $\lambda_{n-1},\lambda_n$ into, respectively, $-\lambda_n$ and $-\lambda_{n-1}$ in type $D_n$. Denote by $w\in W$ the element such that $w(\lambda+\varrho)$ is a strictly decreasing sequence of integers (or half integers), with $\lambda_n\geq 0$ in type $B_n$ and $\lambda_{n-1}+\lambda_n\geq 1$ in type $D_n$. Then, again by the Bott-Borel-Weil theorem, we have that $H^{l(w)}(X,\cE_\lambda)=V_{w(\lambda+\varrho)-\varrho}$ and $H^i(X,\cE_\lambda)=0$ for all $i\neq l(w)$, where $l(w)$ is the \emph{length} of the element $w$ (i.e., the minimal number of simple reflections needed to obtain $w$).
\end{remark}

\begin{notation}
We will denote by $\cU$ and $\cU^\perp/\cU$ the tautological bundles of rank, respectively, $2$ and $m-4$ on $OG(2,m)$. Here $\cU^\perp/\cU$ is a subbundle of the quotient tautological bundle whose vectors are orthogonal to elements in $\cU$. In the weight notation $\cU=[0,-1;0,\dots,0]$, $\cL=\det(\cU^\vee)=[1,1;0,\dots,0]$ and $\cU^\perp/\cU=[0,0;1,0,\dots,0]$. Notice that $\cU^\perp/\cU$ is self-dual.
\end{notation}
The cotangent bundle fits in the short exact sequence $0\to \cO_X(-1)\to \Omega_X\to \cU\otimes (\cU^\perp/\cU) \to 0$, from which we see that $\cF^\vee=\cU\otimes (\cU^\perp/\cU) =[0,-1;1,0,\dots,0]$. We want to compute the cohomology of $\wedge^p \cF^\vee$ and $\wedge^p \cF^\vee\otimes \cL$, but in order to do so we need to be able to express these bundles as direct sums of $\SO(m)$-homogeneous irreducible bundles. As before, let us denote by $S_\bullet$ the Schur functor with weight $\bullet$. If $\lambda=(\lambda_1\geq \cdots \geq \lambda_n)$ is a partition, we will denote by $\lambda'$ the \emph{conjugate partition} of $\lambda$, i.e. the partition whose Young diagram is obtained by a reflection along the main diagonal of the Young diagram of $\lambda$.
\begin{lemma}
We have a decomposition:
$$ 
\wedge^p \cF^\vee=\bigoplus_{0\leq i\leq j, i+j=p}S_{j,i}\cU \otimes S_{2^i,1^{j-i}}(\cU^\perp/\cU) \:\:\: \mbox{ for } \:\:\: 0\leq p\leq 2m-8.
$$
\end{lemma}

\begin{proof}
    This is a direct application of the skew Howe duality for the exterior power of a tensor product (see \cite[Theorem 4.1.1]{howe}) which, if $\lambda'$ is the conjugate partition of $\lambda$ and $|\lambda|=\sum_i \lambda_i$, reads as follows
    \begin{equation*}
        \wedge^p (U\otimes W)=\bigoplus_{|\lambda|=p} S_\lambda U \otimes S_{\lambda'}W. \qedhere
    \end{equation*} 
\end{proof}
In terms of weights, $S_{j,i}\cU=[-i,-j;0,\dots,0]$ and, in order to achieve the desired cohomology computation, we also need to express $S_{2^i,1^{j-i}}(\cU^\perp/\cU)$ in this same language. Consider first $\cU^\perp/\cU$. Being an irreducible vector bundle, it corresponds to a certain irreducible representation of the Levi factor, isomorphic to $\SL(3)\times \SO(m-4)$, of the parabolic subgroup $P_2\subset G$ defining $X=G/P_2$. Moreover, this representation is trivial as a $\SL(3)$-representation and it is the standard $\SO(m-4)$-representation corresponding to the weight $[1,0,\dots,0]$ (notice that in this case there are $n-2$ entries in the sequence). 

As a consequence, our goal is to understand the decomposition of the Schur module $S_{2^i,1^{j-i}}[1,0,\dots,0]$ in direct sums of irreducible $\SO(m-4)$-representations. We will do this in two steps. First, we will use branching rules from $\GL(m-4)$ to $\O(m-4)$ to decompose $S_{2^i,1^{j-i}}[1,0,\dots,0]$ in $\O(m-4)$-representations, then we will use branching rules from $\O(m-4)$ to $\SO(m-4)$ to obtain the decomposition in $\SO(m-4)$-representations. We will follow \cite{flagged} for notations and results. Finally, we will put everything together and compute the cohomology of both $\wedge^p \cF^\vee$ and $\wedge^p \cF^\vee\otimes \cL$.\medskip

When $G$ is one of the three groups $\SO(m-4)$, $\O(m-4)$ and $\GL(m-4)$, and $\lambda$ is the highest weight of an irreducible $G$-representation, we write such a representation by $V_\lambda^G$.
We already described the set of weights for $\GL(m-4)$ and $\SO(m-4)$. To present the weights for $\O(m-4)$, we recall that irreducible $\O(m-4)$-representations are in bijection with decreasing sequences of non-negative integers $\lambda=[\lambda_1,\dots,\lambda_{m-4}]$ such that $\lambda_1'+\lambda_2'\leq n-2$.

\begin{lemma}
Let $i,j$ and $m$ be integers with $0\leq i \leq j$, $i+j \leq 2m-8$ and put $p = i+j$.
Then:
$$
S_{2^i,1^{j-i}}V_{[1,0,\dots,0]}^\GL = \bigoplus_{\max(p-m+4,0) \leq \delta \leq i}V_{[2^{i-\delta},1^{j-i},0^{m-j+\delta-4}]}^\O.
$$
\end{lemma}

\begin{proof}
    This is an application of the branching rule described in \cite[Theorem 4.10]{flagged}, so we will adopt the same notation. It will be sufficient to explain how to recover the statement of our lemma from the cited theorem. To do so, we will go back and forth from \cite{flagged} to unravel this branching result. 
    
    Indeed, the results in the aforementioned paper allow to compute the Littlewood-Richardson type coefficients $\overline{c}^\lambda_{\eta\mu}$ appearing in the decomposition:
    $$ 
    S_\lambda V_{[1,0,\dots,0]}^\GL= \bigoplus_{\mu\in \cP^\O}(\sum_{\eta\in \cP^{(2)}} \overline{c}^\lambda_{\eta\mu} )V_\mu^\O.
    $$
    Here we have denoted by $\cP^\O$ the set of partitions corresponding to weights of $\O(m-4)$, defined by $\mu_1'+\mu_2'\leq m-4$, and by $\cP^{(2)}$ the set of even partitions, i.e., non-increasing sequences satisfying $\eta_u\in 2\ZZ$ for any $u$. As before, $\mu'$ denotes the conjugate partition of $\mu$. From now on, we set $\lambda=[2^{i},1^{j-i}]$, i.e., we focus on the partition in the statement. In this case, $\lambda'=[i,j,0,\dots,0]$. Furthermore, we denote by $LR^\lambda_{\eta\mu}$ the set of Littlewood-Richardson tableaux of shape $\lambda/\eta$ with content $\mu$. Then $\overline{c}^\lambda_{\eta\mu}$ is the cardinality of the subset $\overline{LR}^{\lambda'}_{\eta' \mu'}\subset LR^{\lambda'}_{\eta'\mu'}$ defined in \cite[Section 4.1]{flagged}. 
    Let $ LR^{\lambda'}_{\eta'\mu'}\neq \emptyset$, then $\eta'\subset \lambda'$ implies that $\eta'=[\delta,\delta,0,\dots,0]$ where $0\leq \delta\leq i$, and $\mu'=[j-\delta,i-\delta,0,\dots,0]$, from which we can recover $\eta$ and $\mu$ since the conjugation of Young tableaux is an involution. Recall that we need to impose that $\mu_1'+\mu_2'=j+i-2\delta\leq m-4$, but we will see that this condition is unnecessary. If we denote by $c^{\lambda'}_{\eta'\mu'}$ the cardinality of $LR^{\lambda'}_{\eta'\mu'}$, by the (classical) Littlewood-Richardson rule $c^{\lambda'}_{\eta'\mu'}=1$ if and only if $\eta'=[\delta,\delta,0,\dots,0]$, $0\leq \delta\leq i$ and $\mu'=[j-\delta,i-\delta,0,\dots,0]$, and $c^{\lambda'}_{\eta'\mu'}=0$ otherwise.
    
    The statement of the lemma will then follow if we are able to show that $\overline{LR}^{\lambda'}_{\eta' \mu'}= LR^{\lambda'}_{\eta'\mu'}\neq \emptyset$ if and only if $p-\delta\leq m-4$. Let $\eta^{rev}:=[0,\dots,0,2,\dots,2]$ be the reversed sequence of the sequence $\eta$. Let us suppose that $S\in LR^{\lambda'}_{\eta'\mu'}\neq \emptyset$, then $S$ is a Young tableaux made of two rows of respectively $j-\delta$ blocks and $i-\delta$ blocks. The numbers in the first row will all be equal to one and equal to two in the second row. Thus, always in the notation of \cite{flagged}, $a=j-\delta$, $b=i-\delta$, $(s_1,\dots,s_a)=(1,\dots,1)$ and $(t_1,\dots,t_b)=(2,\dots,2)$. Define 
    $$
    r :=
    \left\{
    \begin{array}{ll}
    m-4-j+\delta & \mbox{if} \:\:\: m-4<2j-2\delta \\
    j-\delta & \mbox{if} \:\:\: m-4\geq 2j-2\delta
    \end{array}
    \right.
    $$
    and
    $$ 
    m_\iota:=\max\{ k\mid \eta_k^{rev}\in X_\iota,\eta_k^{rev}=0  \} \:\:\: \mbox{ for } \:\:\: \iota = 1, \ldots, j-\delta,
    $$
    where $X_\iota$ is defined as:
    $$
    \begin{array}{l} X_\iota = \left\{ \begin{array}{ll}  \{\eta^{rev}_\iota , \dots , \eta^{rev}_{2\iota -1} \} \setminus \{\eta^{rev}_{m_{\iota+1}} , \dots , \eta^{rev}_{m_p}\}  & \textrm{ if } 1\leq \iota\leq r, \\ \{\eta^{rev}_\iota , \dots , \eta^{rev}_{n-p+\iota } \} \setminus \{\eta^{rev}_{m_{\iota+1}} , \dots , \eta^{rev}_{m_p}\}  & \textrm{ if } r< \iota\leq p.  \\ \end{array}\right. \end{array}
    $$
    With this definition, one obtains: 
    $$
    \left\{
    \begin{array}{rcl}
    m_1 & = & m-4-j+1 \\
        & \vdots \\
    m_{j-\delta-1} & = & m-4-\delta-1 \\
    m_{j-\delta} & = & m-4-\delta
    \end{array}
    \right. .
    $$
    Set $f_u$ to be the $u$-th smallest integer in $\{u+1,u+2,\dots,m-4\}\setminus \{m_{u+1},\dots,m_{j-\delta}\}$. Then, by \cite[Theorem 4.10]{flagged}, 
    $$
    S\in \overline{LR}^{\lambda'}_{\eta' \mu'} \iff 2=t_u> \eta^{rev}_{f_u}\mbox{ for }u=1,\dots,i-\delta.
    $$
    Since $\eta^{rev}_k=0$ for $k\leq m-4-\delta$ and $\eta^{rev}_k=2$ for $k\geq m-3-\delta$, and since the sequence $m_1-1,m_2-2,\dots,m_{j-\delta}-j+\delta$ is constant, the above condition is equivalent to 
    \begin{equation*}
        f_{i-\delta}<m-4-\delta+1 \iff 2i-2\delta+j-i<m-4-\delta+1 \iff p-\delta \leq m-4. \qedhere
    \end{equation*}
\end{proof}

\begin{lemma}

Let $i,j$ and $m$ be integers with $0\leq i \leq j$, $i+j \leq 2m-8$ and put $p = i+j$.
Then:
$$ 
S_{2^{i},1^{j-i}}V_{[1,0,\dots,0]}^\GL = \bigoplus_{\max(p-m+4,0) \leq \delta \leq i} R_{\delta}^{i,j}
$$
with
$$
R_{\delta}^{i,j}=
\left\{
\begin{array}{ll}
V_{[2^{i-\delta},1^{j-i},0^{n-j+\delta-2}]}^\SO & \mbox{if }\:\:\: \delta > j-n-h+2\vspace{2mm}\\
V_{[2^{i-\delta},1^{m-4-j-i+2\delta},0^{j-\delta-n+2}]}^\SO & \mbox{if }\:\:\: \delta < j-n-h+2\vspace{2mm}\\
V_{[2^{i-\delta},1,\dots,1,1]}^\SO\oplus V_{[2^{i-\delta},1,\dots,1,-1]}^\SO & \mbox{if }\:\:\: \delta=j-n-h+2 \:\:\: \mbox{ and }\:\:\: i<j\vspace{2mm}\\
V_{[2,\dots,2,2]}^\SO\oplus V_{[2,\dots,2,-2]}^\SO & \mbox{if }\:\:\: \delta=j-n-h+2 \:\:\: \mbox{ and }\:\:\: i=j
\end{array}
\right. .
$$
\end{lemma}

\begin{remark}
The last two cases in the definition of $R_{\delta}^{i,j}$ only appear when $h=0$, i.e., in type $D_n$.
\end{remark}

\begin{proof}
    This is a direct application of the branching rules from $\O(m-4)$ to $\SO(m-4)$ (see \cite{branching}). Indeed, these branching rules imply that, if $\mu=[2^{i-\delta},1^{j-i},0^{m-j+\delta-4}]^\O \in \cP^\O$ with $p-\delta\leq m-4$ then:
    \begin{itemize}
        \item[\_] if $2j-2\delta<m-4$, then $V^\O_\mu\cong V^\SO_\mu$;
        \item[\_] if $2j-2\delta>m-4$, then $V^\O_\mu\cong V^\SO_{\nu}$, where \[\nu=[2^{i-\delta},1^{m-4-p+2\delta},0^{p-i-\delta-n+2}]^\SO;\]
        \item[\_] if $2j-2\delta=m-4$ (so $m$ is even) and $i<j$, then $V^\O_\mu\cong V^\SO_\nu\oplus V^\SO_{\nu'}$, where \[\nu=[2^{i-\delta},1,\dots,1]^\SO, \qquad \nu'=[2^{i-\delta},1,\dots,1,-1]^\SO;\]
        \item[\_] if $2j-2\delta=m-4$ (so $m$ is even) and $i=j$, then $V^\O_\mu\cong V^\SO_\nu\oplus V^\SO_{\nu'}$, where \[\nu=[2,\dots,2]^\SO, \qquad \nu'=[2,\dots,2,-2]^\SO.\] \qedhere
    \end{itemize}
\end{proof}
We are now ready to compute the cohomology of $\wedge^p \cF^\vee$ and $\wedge^p \cF^\vee\otimes \cL$ by applying the Bott-Borel-Weil theorem in type $B_n$. Recall that, in this case, the non-vanishing cohomology of $\Omega^p_X$ for $p\leq m-4$ is given by $H^p(X,\Omega^p_X)\cong \bk^{\lfloor p/2\rfloor +1}$. Moreover, $\fso_{m}=V^{\SO(m)}_{\varpi_2}=[1,1,0,\dots,0]^{\SO(m)}$ and the quasi-minuscule representation is given by the standard representation $V^{\SO(m)}_{\varpi_1}=[1,0,\dots,0]^{\SO(m)}$.

\begin{proposition}
\label{res_jac_bn}
Let $X$ be the adjoint variety of type $B_n$, $m=2n+1$, $h=1/2$ and $e=m-4$. All cohomology groups of $\wedge^p \cF^\vee$ and $\wedge^p \cF^\vee\otimes \cL$ for $0\leq p\leq 2m-8$ vanish except for the following ones:
\begin{itemize}
\item $H^p (X,\wedge^p \cF^\vee)\cong H^p(X, \Omega_X^p) $ for $p\leq e$, where the isomorphism is induced by the surjection $\Omega^p_X\to \wedge^p \cF^\vee $;
\item  $H^{0} (X,\cL)\cong \fso_{m} $ and $H^{n-2} (X,\wedge^{n-1}\cF^\vee\otimes \cL)\cong V_{\varpi_1}^{\SO(m)}$;
\item $ H^{p-2} (X,\wedge^p \cF^\vee\otimes \cL)\cong H^{p-2}(X, \Omega_X^{p-2}) $ for $2\leq p\leq e-1$. Moreover, the terms $ H^{p-2}(X,\Omega_X^{p-2}) $ are the images of the maps induced in cohomology by the embeddings $d\theta\wedge (\bullet):\wedge^{p-2}\cF^\vee \to \wedge^p \cF^\vee\otimes \cL$.
\end{itemize}
\end{proposition}

In the following and the subsequent proof, as already done in the previous sections, we will denote an irreducible homogeneous bundle by its highest weight. 

\begin{proof}
Let us write $\mu_{i,j,\delta}:=[2^{i-\delta},1^{j-i},0^{n-j+\delta-2}]^\SO$ and $\nu_{i,j,\delta}:=[2^{i-\delta},1^{m-4-p+2\delta},0^{p-i-\delta-n+ 1}]^\SO$. Putting together the above lemmas we have, for $0 \leq p \leq 2m-8$:
$$ 
\wedge^p \cF^\vee = \bigoplus_{\substack{\max(p-m+4,0)\leq \delta\leq i\leq j\leq m-4 \\ i+j=p}} W_{i,j,\delta} ,
$$
with 
\[W_{i,j,\delta}:=
\left\{
\begin{array}{ll}
     [-i,-j;\mu_{i,j,\delta}] &  \text{if } \delta > (2j-m+4)/2 \\
     {[-i,-j;\nu_{i,j,\delta}]} & \text{if } \delta < (2j-m+4)/2
\end{array}
\right. 
\]
Also, we get:
\[ \wedge^p \cF^\vee\otimes \cL = \bigoplus_{\substack{\max(p-m+4,0)\leq \delta\leq i\leq j\leq m-4 \\ i+j=p}}W'_{i,j,\delta} \]
with \[
W'_{i,j,\delta}:=
\left\{
\begin{array}{ll}
     [-i+1,-j+1;\mu_{i,j,\delta}] &  \text{if } \delta > (2j-m+4)/2 \\
     {[-i+1,-j+1;\nu_{i,j,\delta}]} & \text{if } \delta < (2j-m+4)/2
\end{array}
\right.\]

\begin{description}
        \item[$\wedge^p \cF^\vee$] Let us begin with the terms of the form $[-i,-j;\mu]$. We have
        \begin{small}
        \begin{align*}
        [-i,-j;\mu]+\varrho&=[n-i-h,n-j-h-1;n-h,\dots\\
    &\ldots,n-i+\delta+h,n-i+\delta-h-1,\dots\\
    &\ldots, n-j+\delta-h,n-j+\delta-2-h,\dots,h].
    \end{align*}
    \end{small}
    In order to have non-vanishing cohomology, either $n-i-h=n-i+\delta-h$ or $n-i-h\leq n-j+\delta-1-h$. In the latter case we get $n-j-h-1\leq -n+i-\delta+h$, which implies $p-\delta\geq m-2$, a contradiction. In the former case we get $\delta=0$ and $n-j-h-1=n-j+\delta-h-1$, which means that we have cohomology isomorphic to $\bk$ at degree $i-\delta+j-\delta=i+j=p$ for $p\leq m-4=e$. A similar computation gives that $[-i,-j;\nu]$ has only cohomology isomorphic to $\bk$ at degree $p$ for $\delta=0$ and $p\leq e$. Putting together both contributions, $\wedge^p \cF^\vee$ has only cohomology in degree $p$. To compute the dimension of this cohomology, notice that we have a one-dimensional contribution for any pair $(i,j)$ such that $i+j=p$, $i\leq j$. A straightforward computation shows that this dimension is thus $\lfloor p/2 \rfloor +1$, as expected.
    \item[$\wedge^p \cF^\vee\otimes \cL$] Let us begin with the terms of the form $[-i+1,-j+1;\mu]$. We have
\begin{small}
    \begin{align*}
        [-i+1,-j+1;\mu]+\varrho=&[n-i+h,n-j-h;n-h,\ldots,\\
        &\dots,n-i+\delta+h,n-i+\delta-h-1,\dots,\\
        &\ldots,n-j+\delta-h,n-j+\delta-2-h,\dots,h].
    \end{align*}
\end{small}    
    In order to have non-vanishing cohomology, either $i=0$ or $n-i+h=n-i+\delta-h$ or $n-i+h\leq n-j+\delta-1-h$. In the last case we get $n-j-h\leq -n+i-\delta+h$, which implies $p-\delta\geq m-3$, a contradiction. In the first case we get $\delta=0$ and there is possibility: $j=0$, $p=0$, the cohomology is in degree $0$ and it is isomorphic to $\fso_m$. In the middle case we get $\delta=1$, $i\geq 1$ and we have cohomology isomorphic to $\bk$ at degree $i-\delta+j-\delta=i+j-2=p-2$ for $p\leq e-1$, or cohomology isomorphic to $V_{\varpi_1}^{\SO(m)}$ when $i=\delta=0$ and $j=n-1=p-1$. A similar computation gives that $[-i+1,-j+1;\nu]$ has only cohomology isomorphic to $\bk$ at degree $p-2$ for $\delta=1$ and $p\leq e-1$. Putting together both contributions, $\wedge^{p} \cF^\vee\otimes \cL$ has only cohomology in degree $0$ (when $p=0$), $n-1$ (when $p=n$) or $p$. Apart from the terms $\fso_m$ and $V_{\varpi_1}^{\SO(m)}$, notice that we have a one-dimensional contribution for any pair $(i,j)$ such that $i+j=p$, $1\leq i\leq j$. A straightforward computation shows that this dimension is thus $\lfloor p/2 \rfloor$, as expected.
\end{description}

The argument about the cohomology groups induced by $\Omega_X^i\to \wedge^i \cF^\vee \to 0$ and $d\theta\wedge(\bullet)$ are the same as in the proof of Proposition \ref{lem_coh_An}.
\end{proof}

Let us now compute the cohomology of $\wedge^p \cF^\vee$ and $\wedge^p \cF^\vee\otimes \cL$ by applying the Bott-Borel-Weil theorem in type $D_n$. Recall that in type $D_n$ the non-vanishing cohomology of $\Omega^p_X$ for $p\leq e=m-4$ is given by $H^p(X, \Omega^p_X)\cong \bk^{\lfloor p/2\rfloor +1+ \delta_{p\geq n-2} }$. Moreover $\fso_{m}=V^{\SO(m)}_{\varpi_2}=V_{[1,1,0,\dots,0]}^{\SO(m)}$.

\begin{proposition}
\label{res_jac_dn}
Let $X$ be the adjoint variety of type $D_n$, $m=2n$, $h=0$, $e=m-4$. All cohomology groups of $\wedge^p \cF^\vee$ and $\wedge^p \cF^\vee\otimes \cL$ for $0\leq p\leq 2m-8$ vanish except for the following ones:
\begin{itemize}
\item $H^p (X,\wedge^p \cF^\vee)\cong H^p(X, \Omega_X^p) $ for $p\leq e$, where the isomorphism is induced by $\Omega^p_X\twoheadrightarrow \wedge^p \cF^\vee $;
\item $H^{0} (X,\cL)\cong \fso_{m} $;
\item $ H^{p-2} (X,\wedge^p \cF^\vee\otimes \cL)\cong H^{p-2}(X, \Omega_X^{p-2}) $ for $2\leq i\leq e-1$. Moreover, the terms $ H^{p-2}(\Omega_X^{p-2}) $ are the images of the maps induced in cohomology by the embeddings $d\theta\wedge (\bullet):\wedge^{p-2}\cF^\vee \hookrightarrow \wedge^p \cF^\vee\otimes \cL$.
\end{itemize}
\end{proposition}

\begin{proof}
Define the following
\begin{align*}
    \mu_{i,j,\delta}&:=[2^{i-\delta},1^{j-i},0^{n-j+\delta-2}]^\SO, \\ \nu_{i,j,\delta}&:=[2^{i-\delta},1^{m-4-p+2\delta},0^{p-i-\delta-n+2}]^\SO,\\ \alpha^\pm_{i,\delta}&:=[2^{i-\delta},1,\dots,1,\pm 1]^\SO,\\ \beta^\pm&:=[2,\dots,2,\pm 2]^\SO.
\end{align*} 
Putting together the above lemmas we have, for $0\leq p\leq 2m-8$:
$$
\wedge^p \cF^\vee = \bigoplus_{\substack{\max(p-m+4,0)\leq \delta\leq i\leq j\leq m-4 \\ i+j=p}} W''_{i,j,\delta} 
$$
with 
$$
W''_{i,j,\delta}:=
\left\{
\begin{array}{ll}
     [-i,-j;\mu_{i,j,\delta}] &  \text{if } \delta > (2j-m+4)/2 \\
     {[-i,-j;\nu_{i,j,\delta}]} & \text{if } \delta < (2j-m+4)/2 \\
     {[-i,-j;\alpha^+_{i,\delta}]\oplus [-i,-j;\alpha^-_{i,\delta}]} & \text{if } \delta = (2j-m+4)/2 \text{ and } i<j\\
     {[-i,-j;\beta^+]\oplus [-i,-j;\beta^-]} & \text{if } \delta = (2j-m+4)/2 \text{ and } i=j\\
\end{array}
\right. ,
$$
and $\wedge^p \cF^\vee\otimes \cL$ has a similar expression. The cohomology of the terms $[-i,-j;\mu_{i,j,\delta}]$ and $[-i,-j;\nu_{i,j,\delta}]$ for $\wedge^p \cF^\vee$ and $[-i+1,-j+1;\mu]$ and $[-i+1,-j+1;\nu]$ for $\wedge^p \cF^\vee\otimes \cL$ are computed exactly in the same way as in the proof of Proposition \ref{res_jac_bn} by choosing $h=0$. The only difference will be that the term $V_{\varpi_1}^{\SO(m)}$ does not appear in this case. Thus, we deal only with the cohomology of the remaining terms. 
\begin{description}
    \item[$\wedge^p \cF^\vee$] Proceeding as in the proof of Proposition \ref{res_jac_bn}, one shows that $[-i,-j;\alpha^+_{i,\delta}]$ and $[-i,-j;\alpha^-_{i,\delta}]$ have cohomology isomorphic to $\bk$ only for $\delta=0$ in degree $p$. Similarly, $[-i,-j;\beta^+]$ and $[-i,-j;\beta^-]$ have cohomology isomorphic to $\bk$ only for $\delta=0$ in degree $p$. Since $\delta=0$, these contributions only appear for $j=n-2$ and $p\geq n-2$. Combining them with the contributions coming from the terms $[-i,-j;\mu_{i,j,\delta}]$ and $[-i,-j;\nu_{i,j,\delta}]$, one obtains that the dimension of the cohomology of $\wedge^p \cF^\vee$ in degree $p$ is equal to: $\lfloor p/2\rfloor +1$ if $p<n-2$ and $\lfloor p/2\rfloor +2$ if $p\geq n-2$, as expected. 
    \item[$\wedge^p \cF^\vee\otimes \cL$] The determination of the cohomology of this bundle does not present any novelty with respect to the above computations and the ones done in Proposition \ref{res_jac_bn}, so we omit it.
\end{description}

The argument about the cohomology groups induced by $\Omega_X^i\to \wedge^i \cF^\vee \to 0$ and $d\theta\wedge(\bullet)$ are the same as in the proof of Proposition \ref{lem_coh_An}.
\end{proof}
    
\subsection{Exceptional types}
 \label{exc_types}   
  In this section, we treat the simple exceptional linear algebraic groups $G_2$, $F_4$, $E_6$, $E_7$, $E_8$.
 
\subsubsection{Type $G_2$.} In this case $X\subset \PP(\fg_2)\cong\PP(V_{\varpi_1})$ is a Fano fivefold of index $3$. It is isomorphic to the quotient $G_2/P_1$. We have $\cL=\cE_{\varpi_1}$, $\cF=\cE_{-\varpi_1+3\varpi_2}$ and $\cF^\vee=\cE_{\alpha_1}=\cE_{-2\varpi_1+3\varpi_2}$. The quasi-minuscule representation is $\hat{\fg}_2=V_{\varpi_2}$.

\subsubsection{Type $F_4$.} In this case $X\subset \PP(\ff_4)\cong\PP(V_{\varpi_1})$ is a Fano $15$-fold of index $8$. It is isomorphic to the quotient $F_4/P_1$. We have $\cL=\cE_{\varpi_1}$, $\cF=\cE_{-\varpi_1+\varpi_2}$ and $\cF^\vee=\cE_{\alpha_1}=\cE_{-2\varpi_1+\varpi_2}$. The quasi-minuscule representation in this case is $\hat{\ff}_4=V_{\varpi_4}$. 

\begin{proposition}
\label{res_jac_g2}
Let $X$ be the adjoint variety of type $G_2$ or $F_4$. All cohomology groups of $\wedge^i \cF^\vee$ and $\wedge^i \cF^\vee\otimes \cL$ vanish except for the following ones:
\begin{itemize}
\item $H^i (X,\wedge^i \cF^\vee)\cong H^i(X, \Omega_X^i) $ for $i\leq e$, where the isomorphism is induced by the surjection $\Omega^i_X\to \wedge^i \cF^\vee $;
\item $H^{0} (X, \cL)\cong \fg $, and: in type $G_2$ $H^{1} (X,\wedge^2\cF^\vee\otimes \cL)\cong  \hat{\fg}$, while in type $F_4$ $H^{2} (X,\wedge^3\cF^\vee\otimes \cL)\cong  \hat{\fg}$;
\item $ H^{i-2} (X,\wedge^i \cF^\vee\otimes \cL)\cong H^{i-2}(X, \Omega_X^{i-2}) $ for $2\leq i\leq e-1$. Moreover, the terms $ H^{i-2}(X, \Omega_X^{i-2}) $ are the images of the maps induced in cohomology by the embeddings $d\theta\wedge (\bullet):\wedge^{i-2}\cF^\vee \to \wedge^i \cF^\vee\otimes \cL$.
\end{itemize}
\end{proposition}

\subsubsection{Type $E_6$.} In this case $X\subset \PP(\fe_6)\cong\PP(V_{\varpi_2})$ is a Fano 21-fold of index $11$. It is isomorphic to the quotient $E_6/P_2$. We have $\cL=\cE_{\varpi_2}$, $\cF=\cE_{-\varpi_2+\varpi_4}$ and $\cF^\vee=\cE_{\alpha_2}=\cE_{-2\varpi_2+\varpi_4}$.
    
\subsubsection{Type $E_7$.} In this case $X\subset \PP(\fe_7)\cong\PP(V_{\varpi_1})$ is a Fano 33-fold of index $17$. It is isomorphic to the quotient $E_7/P_1$. We have $\cL=\cE_{\varpi_1}$, $\cF=\cE_{-\varpi_1+\varpi_3}$ and $\cF^\vee=\cE_{\alpha_1}=\cE_{-2\varpi_1+ \varpi_3}$.
    
\subsubsection{Type $E_8$.} The $E_8$-adjoint variety $X\subset \PP(\fe_8)\cong\PP(V_{\varpi_8})$ is a Fano 57-fold of index $29$. It is isomorphic to the quotient $E_8/P_8$. We have $\cL=\cE_{\varpi_8}$, $\cF=\cE_{-\varpi_8+\varpi_7}$ and $\cF^\vee=\cE_{\alpha_8}=\cE_{-2\varpi_8+ \varpi_7}$.

\begin{proposition}
    \label{res_jac_e6}
    Let $X$ be the adjoint variety of type $E_6$, $E_7$, $E_8$. All cohomology groups of $\wedge^i \cF^\vee$ and $\wedge^i \cF^\vee\otimes \cL$ vanish except for the following ones:
\begin{itemize}
    \item $H^i (X,\wedge^i \cF^\vee)\cong H^i(X,\Omega_X^i) $ for $i\leq e$, where the isomorphism is induced by the surjection $\Omega^i_X\to \wedge^i \cF^\vee $;
    \item $H^{0} (X,\wedge^0\cF^\vee\otimes \cL)\cong \fg $;
    \item $ H^{i-2} (X,\wedge^i \cF^\vee\otimes \cL)\cong H^{i-2}(X,\Omega_X^{i-2})  $ for $2\leq i\leq e-1$. Moreover, the terms $ H^{i-2}(X, \Omega_X^{i-2}) $ are the images of the maps induced in cohomology by the embeddings $d\theta\wedge (\bullet):\wedge^{i-2}\cF^\vee \to \wedge^i \cF^\vee\otimes \cL$.
    \end{itemize}
    \end{proposition}

  \begin{proof}[Proof of Propositions \ref{res_jac_g2},
  \ref{res_jac_e6}] 
  The group $G_2$ case can be dealt with by hand. For the groups $F_4$, $E_6$, $E_7$ and $E_8$ we used a Python script using \cite{Lie} in order to compute the cohomology groups of the relevant homogeneous bundles through the Bott-Borel-Weil theorem. The argument about the cohomology groups induced by $\Omega_X^i\to \wedge^i \cF^\vee \to 0$ and $d\theta\wedge(\bullet)$ are the same as in the proof of Proposition \ref{lem_coh_An}.
  \end{proof}

\subsection{Free resolution of the Jacobian ideal of adjoint discriminant}
    
Now we are ready to state and prove the main results of this section.
Namely, we provide a uniform expression for a locally free resolution of the restricted Jacobian ideal of discriminants of adjoint varieties. 

\subsubsection{The cup product with the hyperplane class}

From the exact sequence \eqref{eq_omegap_F}, for $i=1,\ldots,N$, we can compute the cohomology of $\Omega^i_X\otimes \cL$ from the cohomology of $\wedge^{i-1}\cF^\vee$ and $\wedge^i \cF^\vee\otimes \cL$. Recall that $K^{i-1}$ and $C^i$ are, respectively, the kernel and cokernel of the multiplication by the hyperplane class $H^{i-1}(X, \Omega^{i-1}_X)\to H^i(X, \Omega^i_X)$. Recall moreover that $H^{i-1}(X,\wedge^{i-1}\cF^\vee)\cong H^{i-1}(\Omega^{i-1}_X)$ and that the factor $H^{i-2}(\Omega^{i-2}_X)$ inside $H^{i-2}(X,\wedge^i \cF^\vee\otimes \cL)$ is induced by the contact form $d\theta$ (Propositions \ref{lem_coh_An}, \ref{res_jac_bn}, \ref{res_jac_dn}, \ref{res_jac_g2}, \ref{res_jac_e6}). We can restrict the connecting homomorphism $H^{i-2}(X,\wedge^i \cF^\vee\otimes \cL) \to H^{i-1}(X,\wedge^{i-1}\cF^\vee)$ induced by the exact sequence \eqref{eq_omegap_F} to obtain a morphism 
$$ \eta_i: H^{i-2}(X, \Omega^{i-2}_X)\to H^{i-1}(X, \Omega^{i-1}_X).$$

\begin{lemma}
\label{lem_max_rank}
The morphism $\eta_i$ is, modulo non-zero scalar, the multiplication by the hyperplane class $[c_1(\cL)]$. In particular, it has maximal rank. 
\end{lemma}
    
\begin{proof}
    The statement about the rank is a consequence of the Lefschetz hyperplane theorem. Moreover, since the maps $\eta_i$ are induced for each $i$ by the map $\eta_2$ by definition of $d\theta\wedge (\bullet)$, it is sufficient to prove the statement for $\eta_2$. In this case the exact sequence \eqref{eq_omegap_F} is just a twist of
    $$ 0\to \cL^\vee \to \Omega_X \to \cF^\vee \to 0 .$$
    The above sequence describes a non-trivial extension, induced by the element in $H^1(X,\cF\otimes \cL^\vee)\cong H^{1}(X,\cF^\vee)\cong H^1(\Omega_X)$ corresponding to $\cL$ (see for instance \cite{bucmor}). Thus the image of $\eta_2$ is $c_1[\cL]$ and as a consequence $\eta_i$ is the multiplication by the hyperplane class.
\end{proof}

\subsubsection{The resolution of the Jacobian ideal and Theorem \ref{main not simply laced}}

Let $G$ be a simple linear algebraic group of rank $n$ and let $\epsilon=0$ (respectively $\epsilon =1$) if $G$ is not of simply laced type (respectively $G$ is of simply laced type). Let $\fg$ be the Lie algebra of $G$ and let $\hat{\fg}$ be the quasi-minuscule representation. Recall that, in the simply laced case, $\hat{\fg}\cong\fg$. Let $j=0$ if $G$ is of simply-laced type. If $G$ is not of simply-laced type, we let $j$ be the smallest short exponent of $G$, see \cite{kostant:lie-group}. In other words, $j=1$ if $G=C_n$, $j=2$ if $G=G_2$, $j=3$ if $G=F_4$ and $j=n-1$ if $G=B_n$. Let $X$ be the adjoint variety of $G$. 
Recall that $(e_1\leq \dots \leq e_n)$ are the exponents of $\fg$ and $s$ is the number of long simple roots of $\fg$.

\begin{theorem} 
\label{thm_res}
    There exists a $G$-equivariant free resolution of the form:
\begin{small}
\[0 \to \begin{array}{c} \bigoplus_{i=1}^s \bU(-e_i-1) \\
\oplus \\
\delta_{0,\epsilon}\hat{\fg}\otimes \bU(-j-1)
\end{array}
    \to
    \begin{array}{c}
         \bigoplus_{i=1}^s \bU(-e_i)  \\
         \oplus \\
         \delta_{0,\epsilon}\hat{\fg}\otimes \bU(-j) \oplus \fg \otimes \bU(-1) 
    \end{array} 
    \to \fg\otimes \bU \to J_D(d-1)\to 0.\]
\end{small}
\end{theorem}

Let us note that Theorem \ref{main not simply laced} follows from the previous result, coupled with \eqref{eq_affine_tangent_jac_modules}.
In terms of sheaves, using Theorem \ref{main not simply laced} and \eqref{relation-T-Der}, we get the following result.
\begin{corollary} Let $G$ be a simple linear algebraic group, not of simply laced type. Let $\Delta$ be the adjoint discriminant of $G$ and set $D=\VV(\Delta)$. Then we have a
$G$-equivariant locally free resolution of the form:
\[
\begin{small}
0 \to \begin{array}{c}
\bigoplus_{i=1}^s \cO_{\PP(\fg)}(-e_i) \\ \oplus \\ \hat \fg \otimes \cO_{\PP(\fg)}(-j)
\end{array}\to
\begin{array}{c}
 \bigoplus_{i=1}^s \cO_{\PP(\fg)}(1-e_i) \oplus  \fg \otimes \cO_{\PP(\fg)}\\
 \oplus \\
 \hat \fg \otimes \cO_{\PP(\fg)}(1-j)
 \end{array}
 \to \cT_\PP\langle D \rangle \to 0.
\end{small}
\]
\end{corollary}

\begin{proof}[Proof of Theorem \ref{thm_res}]
The result is a consequence of Theorem \ref{thm_resolution_pushforward}. In order to apply it, we use the cohomology computations given in Propositions \ref{lem_coh_An}, \ref{res_jac_bn}, \ref{res_jac_dn}, \ref{res_jac_g2}, \ref{res_jac_e6}.
Then, by Lemma \ref{lem_max_rank}, one obtains
\begin{small}
\begin{equation}\label{diag-Ginvresolution}
0 \to \begin{array}{c}
     \oplus_{p\leq e} C^{p}\otimes \bU(-p-2)  \\
     \oplus \\
     \delta_{0,\epsilon}\hat{\fg}\otimes \bU(-j-1)
\end{array} \to 
\begin{array}{c}
     \oplus_{p\leq e} C^{p}\otimes \bU(-p-1) \\
     \oplus \\
     \delta_{0,\epsilon}\hat{\fg}\otimes \bU(-j)\\ \oplus \\\fg \otimes \bU (-1) 
\end{array} \to \fg\otimes \bU \to \rhop_*\pip^*\cL\to 0.
\end{equation}
\end{small}

    The result follows by noticing that, as in the proof of Theorem \ref{thm_res_structure}, $C^i$ is a direct sum of $u_i$ trivial $G$-representations, where $u_i$ is the cardinality of $\{ j\mid \deg(f_j)=i+2\}$.
\end{proof}    

\subsubsection{Identifying the Lie bracket}

Let us end this section by focusing on the $G$-equivariant morphism $\fg\otimes \bU(-1) \to \fg\otimes \bU$. By taking global sections this morphism is given by a trivial $G$-factor of $\fg\otimes \fg \otimes \fg$, where we silently use the Killing form to identify $\fg\cong\fg^\vee$.

\begin{proposition}\label{prop-LieB}
 The morphism $\fg\otimes \bU(-1) \to \fg\otimes \bU$ appearing in the resolution of Theorem \ref{thm_resolution_pushforward} is given by the Lie bracket in $\fg$. 
\end{proposition}

\begin{proof}
Following the proof of Theorem \ref{thm_resolution_pushforward}, the term $\fg\otimes \bU$ comes from the term $H^0(\cL)\otimes \bU$ in Weyman's resolution \eqref{eq_Weyman_res} and $\fg\otimes \bU(-1)$ comes from $H^0(\hat{\Omega}_X\otimes \cL)\otimes \bU(-1)$ . Thus, the morphism $\eta:\fg\otimes \bU(-1) \to \fg\otimes \bU$ is the pushforward, through the second projection from $X\times \fg$, of the morphism $\hat{\Omega}_X\otimes \cL \boxtimes \bU(-1) \to \cL\boxtimes \bU$ in the Koszul complex of Weyman's resolution. By Lemma \ref{lem_iso_xi}, we can identify, due to the contact structure, $\hat{\Omega}_X\otimes \cL \cong \hat{\cT}_X$, so we obtain $$\eta':\hat{\cT}_X\boxtimes \bU(-1) \to \hat{\Omega}_X\otimes \cL\boxtimes \bU(-1) \to \cL\boxtimes \bU.$$ 
Indeed, if $\theta':\hat{\cT}_X \to \cL$ is the map induced by the contact structure $\theta:\cT_X\to \cL$ then $d\theta':\wedge^2 \hat{\cT}_X \to \cL$, and the map $d\theta'$ induces the isomorphism $\hat{\cT}_X\cong \hat{\Omega}_X\otimes \cL$. Thus we get $d\theta'\in H^0(\hat{\cT}_X\otimes \cL)\subset H^0(\hat{\cT}_\PP\otimes \cL)\subset \wedge^2 \fg \otimes \fg$, where in the last equality we have used the Killing isomorphism. Since the contact structure is defined from the Kostant-Kirillov form (see \cite{beauville:fano-contact}), $d\theta'\in \wedge^2 \fg \otimes \fg$ is the Lie bracket (modulo scalar).

Finally notice that $H^0(\hat{\cT}_X)$ is an extension of a trivial factor with $H^0(\cT_X)\cong \fg$. Since the map $\hat{\Omega}_X\otimes \cL \to \cL$ in the Koszul complex is just the contraction (after $X\times \fg$ is identified with the total space of $\hat{\Omega}_\PP\otimes \cL$), we obtain that the restriction $\eta$ of $H^0(\eta'):H^0(\hat{\cT}_X) \otimes \bU(-1) \to H^0(\cL) \otimes \bU$ to $\fg \otimes \bU(-1)\subset H^0(\hat{\cT}_X)\otimes \bU(-1)$ is the Lie bracket. 
\end{proof}
    
   \section{Logarithmic derivations for simply laced adjoint discriminants} 
    \label{section: non invariant}
    
The main goal of this section is to prove Theorem \ref{main simply laced}. The idea is to interpret the results of the previous section in terms of derivations, inspired by the treatment of invariant derivations of \cite{orlik-terao:arrangements}. We use the notation of the introduction: $G$ will be a simple linear algebraic group with a Lie algebra $\fg$ of \textit{simply laced type}. Choosing a maximal torus $T\subset G$, we get a Weyl group  $W:=N_G(T)/Z_G(T)$ acting linearly on the Lie algebra $\fh$ of $T$. We write  $\bU:=\bk[\fg]$, $\bS:=\bk[\fh]$ and recall that Chevalley's restriction theorem gives $\bR:=\bU^G \cong \bS^W$, i.e., in terms of GIT, we have $\fg \git G \simeq \fh /W \simeq \bk^n$ and $\fh$ is a slice for the action of $G$ on $\fg$.
Set $n=\dim(\fh)$.

\subsection{$G$-invariant logarithmic derivations} 
 
   Let $\Der_\bU(-\log(\Delta))$ be the $\bU$-module of logarithmic derivations of the discriminant locus $D = X^\vee = \VV(\Delta)$, where $X\subset \PP(\fg)$ is the $G$-adjoint variety. 
   
   We define $\Der_\bS^W$ as the $W$-invariant $\bS$-derivations and $\Der_\bU^G$ as the $G$-invariant $\bU$-derivations. Since $\bR=\bU^G$, for any $\eta\in \Der_\bU^G$, we have $\eta(\bR)\subset \bR$, so that $\eta$ defines an element $\pi_G(\eta)\in \Der_\bR$. Similarly, since $\bR=\bS^W$, any $\eta\in \Der_\bS^W$, defines an element $\pi_W(\eta)\in \Der_\bR$. The maps $\pi_G:\Der_\bU^G \to \Der_\bR$ and $\pi_W:\Der_\bS^W\to \Der_\bR$ are morphisms of $\bR$-modules. 
   Consider the restriction $\delta:=\Delta|_\fh$. Since $\fg$ is of simply laced type, $\delta$ is a square, and we may write
   $$
   \sqrt \delta = \Pi_{\alpha \in \Phi^+}\delta_\alpha.
   $$ 
   This product defines the Weyl arrangement associated with the Weyl group $W$. Recall that, since $W$ is generated by reflections, the associated Weyl arrangement is the union of the reflecting hyperplanes of all (complex) reflections in $W$. We have: 
   \[
   \Der_\bR(-\log(\delta)):= \{\eta\in \Der_\bR \mid \eta(\delta)\in (\delta)\subset \bR\} \cong \Der_\bR(-\log(\sqrt \delta)).\]
   Seeing $\delta$ as an element in $S$, we can also consider $\Der_\bS(-\log(\delta))$, i.e., the $\bS$-module of logarithmic derivations defined by $\delta$.
   
   Since we are in characteristic zero and since $\Frac(\bS)$ is finite algebraic over $\Frac(\bR)$, for any $\eta\in \Der_\bR\subset \Der_{\Frac(\bR)}$, there exists a unique derivation $\eta_\bS\in \Der_{\Frac(\bS)}^W$ such that $\eta_\bS|_\bR=\eta$. Set:
   \begin{align*}
   \Der_\bR^0:=&\{\eta\in \Der_{\Frac(\bR)} \mid \eta_\bS(\bS)\subset \bS\}=\\
   =&\{\eta\in \Der_{\Frac(\bR)} \mid \eta_\bS\in \Der_\bS\subset \Der_{\Frac(\bS)}\}.
\end{align*}

   Combining several results that appear in \cite[Section 6.3]{orlik-terao:arrangements}, we have that: 
\begin{theorem}
   \label{thm_terao_der0}
       The morphism $\pi_W:\Der_\bS^W \to \Der_\bR$ induces a $\bR$-module isomorphism between $\Der_\bS^W$ and $\Der_\bR^0$. Furthermore, $\Der_\bR^0\cong\Der_\bR(-\log(\delta))$ is a free $\bR$-module and $\Der_\bS(-\log(\delta))\cong  \Der_\bS^W\otimes_\bR \bS$ is a free $\bS$-module.
   \end{theorem}

   Let $\fg=\fh\oplus\bigoplus_\alpha \fg_\alpha$ be a Cartan decomposition of $\fg$ and identify $\fg$ and $\fg^\vee$ via the Killing form. Let $(x_1,\ldots,x_\ell)$ be a Killing-orthonormal basis of $\fg^\vee$ such that $(x_1,\ldots,x_n)$ is a basis of $\fh^\vee$ and for any $n+1\leq i\leq \ell$, there exists a root $\alpha$ such that $x_i\in \fg_\alpha\oplus \fg_{-\alpha}$. 
   Under these assumptions, $\bS=\bU/(x_i)_{i=n+1}^{\ell}$. We will denote by $\bI:=(x_i)_{i=n+1}^{\ell}$ the ideal defining $\bS$.
    Any homogeneous element $\eta\in \Der_\bU$ of degree $d$ can be therefore written as 
   \begin{equation}\label{eq-defeta}
   \eta=\sum_{i=1}^{\ell} \eta_i \frac{\partial}{\partial x_i}, \qquad \mbox{with $\eta_i:=\eta(x_i) \in \bU_d$}.
   \end{equation}
   Moreover, let us denote by $(g_{i,j})_{1 \le i,j \le \ell} $ the matrix associated to the linear action of $g\in G$ on $\fg^\vee$ in the basis $(x_1,\ldots,x_\ell)$, i.e. $g(x_i)=\sum_{j=1}^{\ell} g_{i,j}x_j$. We will prove some technical results which will be needed for our main theorem. 
   
   \begin{lemma}\label{lemma-actionder}
    Let $\eta\in \Der_\bU^G$. Then for, any $g\in G$, 
    \[g(\eta_i)=\sum_{j=1}^{\ell} g_{i,j}\eta_j,\]
    where $\eta_i$ are as defined in (\ref{eq-defeta}).
   \end{lemma}
   
   \begin{proof}
   The action of $g\in G$ on $\Der_\bU$ is given by:
   $$
   (g\cdot \eta)(f):= g\cdot (\eta(g^{-1}\cdot f)).
   $$
   Since, for all $j=1,\ldots,\ell$, 
   $$ 
   g\cdot\left(\frac{\partial}{\partial x_i} (g^{-1})\cdot x_j)\right) = g\cdot\left( \frac{\partial}{\partial x_i} \left(\sum_{k=1}^{\ell} g^{-1}_{j,k}x_k\right)\right)=g^{-1}_{j,i},
   $$
   we deduce that $g\cdot \frac{\partial}{\partial x_i}=\sum_{j=1}^{\ell} g^{-1}_{j,i}\frac{\partial}{\partial x_j}$. Let $\eta\in \Der_\bU^G$, i.e., $g\cdot \eta=\eta$ for any $g\in G$. This means that, for any $g\in G$,
   $$
   \sum_{1\le i,j \le \ell}g(\eta_i)g^{-1}_{j,i}\frac{\partial}{\partial x_j} = \sum_{j=1}^{\ell} \eta_j \frac{\partial}{\partial x_j}.$$
   This proves the claim.
   \end{proof}
   
   \begin{lemma}
    Let $\eta\in \Der_\bU^G$. 
    \begin{itemize}
       \item If $i\leq n$ then $\frac{\partial f}{\partial x_i}|_\fh=\frac{\partial f|_\fh}{\partial x_i}$ for any $f\in \bU$;
       \item if $i\geq n+1$ then $\eta_i|_\fh=0$.
    \end{itemize}
   \end{lemma}
   
   \begin{proof}
   The first equality is due to the fact that restricting to $\fh$ amounts to working modulo the ideal $\bI$ that defines $\bS$ in $\bU$. For the second equality, let us fix a root $\alpha$ such that $y_i:=x_i+x_{i'}\in \fg_\alpha$ and $y_i':=x_i-x_{i'}\in \fg_{-\alpha}$, with $x_i,x_{i'}\in \fg_\alpha\oplus \fg_{-\alpha}$. Moreover, let $t\in T$ be any point in a maximal torus $T\subset G$ stabilizing $\fh$. Then, since the basis $(\{ x_j\}_{j\leq n}, \{y_i\}_{i\geq n+1})$ is compatible with the Cartan decomposition defined by $T$, $T$ acts diagonally on $\fg^\vee$ in this basis. More precisely $t\cdot y_i = \alpha(t)y_i$, where $\alpha$ is seen as a character of $T$. Since, by Lemma \ref{lemma-actionder}, the action of $G$ on $\eta_i$ behaves as the action of $G$ on $x_i$, we also get that $t\cdot (\eta_i\pm\eta_{i'}) = \alpha(t)(\eta_i\pm\eta_{i'})$. Since $T$ acts as the identity on $\fh$ we deduce that $t\cdot (\eta_i|_\fh) = \eta_i|_\fh$ and, combined with its diagonal action, we also deduce that $t \cdot ((\eta_i\pm\eta_{i'})|_\fh)=(t\cdot (\eta_i\pm\eta_{i'}))|_\fh$. By choosing $t\in T$ such that $\alpha(t)\neq 1$, we obtain that $(\eta_i\pm\eta_{i'})|_\fh=0$, i.e., that $\eta_i|_\fh=0$ for any $i\geq n+1$. 
   \end{proof}
   
   Recall that, by the Chevalley–Shephard–Todd theorem, there exists a basis $(F_1,\ldots,F_n)$ of homogeneous $G$-invariant polynomials in $\bU$, i.e., $\bR\cong \bk[F_1,\ldots,F_n]$ and denote by $f_1,\ldots,f_n$ their restrictions to $\bS$. 
   For $i=1,\ldots,n$, let us define $\mu_i\in \Der_\bS$ and $\nu_i\in \Der_\bU$ by
   \begin{equation}\label{eq-defmunu}
   \mu_i:=\sum_{j=1}^n\frac{\partial f_i}{\partial x_j}\frac{\partial }{\partial x_j}, \qquad \nu_i:=\sum_{j=1}^{\ell}\frac{\partial F_i}{\partial x_j}\frac{\partial }{\partial x_j}.
   \end{equation}
   
   \begin{lemma}
   The derivations $\nu_i$ for $i=1,\ldots,n$ are $G$-invariant.
   \end{lemma}
   
   \begin{proof}
   Let us compute $g\cdot \nu_i$, for $i=1,\ldots,n$. We obtain
   $$ (g\cdot \nu_i)(x_k)=\sum_j g\left(\frac{\partial F_i}{\partial x_j}\right) g\left(\frac{\partial (g^{-1}\cdot x_k)}{\partial x_j}\right)=\sum_j g\left(\frac{\partial F_i}{\partial x_j}\right) g^{-1}_{k,j}.  $$
   By the $G$-invariance of $F_i$ we also deduce that
   $$ g\left(\frac{\partial F_i}{\partial x_j}\right)=\left(g\cdot \frac{\partial }{\partial x_j}\right)(g(F_i))=\left(\sum_h g_{j,h}\frac{\partial }{\partial x_h}\right)(F_i)=\sum_h g_{j,h}\frac{\partial F_i}{\partial x_h} .$$
   Combining the two expressions, we have 
   $$ (g\cdot \nu_i)(x_k)=\sum_{j,h}g^{-1}_{k,j}g_{j,h}\frac{\partial F_i}{\partial x_h}=\sum_h \delta_{k,h}\frac{\partial F_i}{\partial x_h}=\frac{\partial F_i}{\partial x_k}=\nu_i(x_k),$$
   which means exactly that $g\cdot \nu_i = \nu_i$ for any $g\in G$.
   \end{proof}
   
   The same proof shows that $\mu_i\in \Der_\bS^W$. Furthermore, it is known that $\mu_1,\ldots,\mu_n$ is an $\bR$-basis of $\Der_\bS^W$, see for instance \cite[Section 2]{Yos}.
   
   \begin{proposition}
    There exists a morphism of $\bR$-modules $\pi:\Der_\bU^G\to \Der_\bS^W$ making the following diagram commutative:
    $$
    \xymatrix{
\Der_\bU^G \ar[d]_-\pi \ar[r]^-{\pi_G}_-\cong &\Der_\bR(-\log(\Delta))\ar[d]^-\cong\\
\Der_\bS^W \ar[r]^-\cong_-{\pi_W} & \Der_\bR(-\log(\delta))} .
    $$
   \end{proposition}

   \begin{proof}
   Let us explicitly define $\pi$. Let $\eta=\sum_i \eta_i \frac{\partial}{\partial x_i} \in \Der_\bU^G$. For any polynomial $f\in \bk[\fg]$, we can define $\pi(\eta)(f|_{\fh}):= \eta(f)|_\fh$. This is well defined because if $f|_\fh=0$ then $f\in \bI$ and, by the above lemmas, $$\eta(f)|_\fh= \sum_i \eta_i|_\fh \frac{\partial f}{\partial x_i}|_\fh=\sum_{i \leq n} \eta_i|_\fh \frac{\partial f}{\partial x_i}|_\fh= \sum_{i\leq n} \eta_i \frac{\partial f|_\fh}{\partial x_i}=0.$$ 
   
   Even though a priori $\pi$ is defined as a morphism $\pi:\Der_\bU^G\to \Der_\bS$, its image is $W$-invariant. Chevalley theorem tells us that there exists an isomorphism between $\Der_\bR(-\log(\Delta))$ and $\Der_\bR(-\log(\delta))$ given by restriction, so, modulo this isomorphism, $\pi_W \circ \pi = \pi_G$. As a consequence, we deduce $\Der_\bU^G\subset \Der_\bU(-\log(\Delta))$.\medskip
    
    Having this set up, we need to show that $\pi$ is an isomorphism to conclude. Let us begin with the injectivity. Let $\theta\in \Der_\bS^W$ and $\eta,\eta'\in \Der_\bU^G$ such that $\pi(\eta)=\pi(\eta')=\theta$ and denote their difference by $\xi:=\eta-\eta'\in \Der_\bU^G$. By linearity $\pi(\xi)=0$. Composing with $\pi_W$ we deduce that, for any $f\in \bR$, $\xi(f)|_\fh=0$. Since $\xi(f)\in \bR$, by Chevalley's Theorem we have that $\xi(f)=0$ for any $f\in \bR$.
    
    Let $\fh'$ be any Cartan subalgebra of $\fg$ for which, as before, we find a basis $(x_1',\ldots,x_\ell')$, write $\xi=\sum_i \xi_i' \frac{\partial}{\partial x_i'}$ and define $\pi':\Der_\bU^G \to \Der_{\bS'}^{W'}$ by restriction. However, since for any $f\in \bR$, $\xi(f)|_{\fh'}=0|_{\fh'}=0$, we have that $\pi_{W'}\circ \pi'(\xi)=0$. On the other hand, $\pi'(\xi)$ is the unique extension of $\pi_{W'}\circ \pi'(\xi)\in \Der_\bR(\delta)$ to $\Der_{\bS'}^{W'}$, since $\Frac(\bS)$ is a finite extension of $\Frac(\bR)$ in characteristic zero. Clearly $0$ extends $0$, so $\pi'(\xi)=0$.
    
    To summarize, we have shown that, for any Cartan subalgebra $\fh'\subset \fg$ and for any $f\in \bU$, we have $\xi(f)|_{\fh'}=0$. Since Cartan subalgebras cover a dense subset of $\fg$, the equality above implies that $\xi(f)=0$ for any $f\in \bU$, i.e., that $\xi=0$. Thus $\pi$ is injective. 
    \medskip
    
    For the surjectivity, consider, for $i=1,\ldots,n$, $\mu_i$ and $\nu_i$ as defined in (\ref{eq-defmunu}). Recall that the $\mu_i$'s define a basis of $\Der_\bS^W$ and notice that, by construction, $\pi(\nu_i)=\mu_i$. Since $\nu_i\in \Der_\bU^G$, for $i=1,\ldots,n$, the $\bR$-morphism $\pi$ is surjective and thus an isomorphism.
   \end{proof}
   
   \begin{proposition}\label{prop-freeinvariantmodule}
    The morphism $\pi_G$ is an isomorphism of $\bR$-modules and there is a $G$-equivariant inclusion $\Der_\bU(-\log(\Delta))^G:=\Der_\bU^G\otimes_\bR \bU \subset \Der_\bU(-\log(\Delta))$, where $\Der_\bU(-\log(\Delta))^G$ is a free $\bU$-module of rank $n$.
   \end{proposition}
   
   \begin{proof}
    The morphism $\pi_G$ is an isomorphism between $\Der_\bU^G$ and $\Der_\bR(-\log(\Delta))$ by the previous lemma. Since $\Der_\bU^G\cong \Der_\bR^0\cong \Der_\bR(-\log(\Delta))$ as $\bR$-modules and $\Der_\bR(-\log(\Delta))$ is a free $\bR$-module of rank $n$, we deduce that $\Der_\bU^G\otimes_\bR \bU$ is a free $\bU$-module of rank equal to $n$.
   \end{proof}

\subsection{Minimality of the resolution for simply laced types}
   In the following we use the identification of the $\bU$-module $\Der_\bU(-\log(\Delta))$ with the kernel of the differential surjective morphism $\fg\otimes \bU \to J_{D}(d-1)$ defining the restricted Jacobian ideal.
  
   \begin{theorem}
   \label{thm_minimality}
   Let $G$ be of simply laced type. Then the resolution in Theorem \ref{thm_res} is minimal.
   \end{theorem}

   \begin{proof}
   In the simply laced types the resolution obtained takes the following form:
 $$   
 0\to \bigoplus_{1\leq i\leq n} \bU(-e_i-1) \to      \bigoplus_{1\leq i\leq n} \bU(-e_i) \oplus \fg\otimes \bU(-1) \to \Der_\bU(-\log(\Delta)) \to 0.    
 $$
This resolution is $G$-equivariant for the natural action of $G$ on $\fg$ and for the trivial action of $G$ on $\bU(j)$ for any $j$. Moreover, $\dim(\bigoplus_{1\leq i\leq n}\bU(-e_i))=\rank(G)=n$. Thus, the $G$-invariant $\bU$-module $\Der_\bU(-\log(\Delta))^G$ is generated by the image of $\bigoplus_{1\leq i\leq n}\bU(-e_i)$ inside $\Der_\bU(-\log(\Delta))$. Furthermore, since $\Der_\bU(-\log(\Delta))^G$ is free of rank $n$ by Proposition \ref{prop-freeinvariantmodule}, the restriction morphism $\bigoplus_{1\leq i\leq n}\bU(-e_i) \to \Der_\bU(-\log(\Delta))$ is injective and induces a fundamental isomorphism 
\begin{equation} \label{exponents and invariant derivations}
    \bigoplus_{1\leq i\leq n}\bU(-e_i) \xrightarrow{\simeq} \Der_\bU(-\log(\Delta))^G
\end{equation}
As a consequence, there cannot be a non-zero constant morphism $\bU  (-j) \to \bU(-j)$, i.e., the resolution is minimal.
   \end{proof}

   We define the module
   $$
   \Der_\bU(-\log(\Delta))_1:= \Der_\bU(-\log(\Delta)) / \Der_\bU(-\log(\Delta))^G.
   $$
   From the proof of Theorem \ref{thm_minimality}, we get the following result.
   \begin{corollary}
   Assume $G$ is of simply laced type. A free resolution of the $\bU$-module $\Der_\bU(-\log(\Delta))_1$ is: 
   $$
   0\to \bigoplus_{i=1}^n \bU(-e_i-1) \to \fg \otimes \bU(-1) \to \Der_\bU(-\log(D))_1 \to 0.
   $$
   \end{corollary}
   
   \begin{remark}
       We believe that this result can be generalised to the non-simply laced case, and even further to the case of \emph{complex} reflection groups and graded Lie algebras. 
   \end{remark}

\subsection{$G$-variant logarithmic derivations}
   
   The restriction to $\fg\otimes \cO(-1)$ of the middle map in the resolution described in (\ref{diag-Ginvresolution}) induces a morphism $$\bad:\fg\otimes \bU(-1) \to \fg\otimes \bU.$$ We want to understand what the kernel of this morphism is.
   
   \begin{lemma}\label{lemma-kerad}
     In the simply laced case, each derivation $$\nu_i=\sum_{j=1}^{\ell}\frac{\partial F_i}{\partial x_j}\frac{\partial }{\partial x_j}\in \fg\otimes \bU(-1), \qquad \text{for $1\leq i\leq n$,}$$ 
     lies in $\ker(\bad)$.
   \end{lemma}
   
   \begin{proof}
   By Lemma \ref{prop-LieB}, the restriction of the middle map in (\ref{diag-Ginvresolution}) to $\fg\otimes\cO(-1)$ is the Lie bracket of $\fg$, i.e., the differential at the identity of the adjoint action of $G$ on $\fg$. The induced morphism $\fg\otimes \bU(-1) \to \fg\otimes \bU$ is thus the differential of the action of $G$ on $\fg\otimes \bU$, which can be identified with $\Der_\bU$. On an element $\nu\in \Der_\bU$, the action $\psi_\nu:G\to \Der_\bU$ is given by $g\mapsto g\cdot \nu$. Since any $\nu_i$, for $1\leq i\leq n$, is a $G$-invariant derivation, we get that $g\cdot \nu_i=\nu_i$ for all $g\in G$; thus the differential of $\psi_{\nu_i}$, along any tangent direction at the identity of $G$, vanishes. This shows that $\nu_i$ is in the kernel of $\bad$.
   \end{proof}
  
We conclude this section by proving Theorem \ref{main simply laced}.
When $G$ is of simply laced type, we show that the module  $\Der_\bU(-\log(\Delta))_1$ is a direct summand of the module of logarithmic derivations and that furthermore it coincides with the adjoint image module $\bA = \IM(\bad)$ of $G$\textit{-variant logarithmic derivations}.

\begin{theorem} \label{diamogli un nome}
Let $G$ be of simply laced type. Then we have isomorphisms $\Der_\bU(-\log(\Delta))_1 \simeq \bA$ and 
$$
\Der_\bU(-\log(\Delta)) \simeq \bA \oplus \Der_\bU(-\log(\Delta))^G.
$$
\end{theorem}
\begin{proof}
The resolution obtained in Theorem \ref{thm_minimality} induces the following commutative diagram, constructed starting from the central vertical exact sequences
\begin{small}
$$
\xymatrix@C-2.5ex{
0 \ar[r] & \ker (\bad) \ar@{^(->}[d] \ar^-\nu[r] & \fg \otimes \bU(-1) \ar@{^(->}[d] \ar[r]^-{\bad} & \fg \otimes \bU \ar@{=}[d] \ar[r] & \coker (\bad) \ar[r] \ar@{->>}[d] & 0\\
0 \ar[r] & \bigoplus_{i=1}^n \bU(-e_i-1) \ar@{->>}[d]^{\beta} \ar[r]^-{(\alpha, \beta)} & \bigoplus_{i=1}^n \bU(-e_i) \oplus \fg \otimes \bU(-1) \ar@{->>}[d] \ar[r]^-{\varphi} & \fg \otimes \bU  \ar[r] & J_D(d-1) \ar[r] & 0\\
& \bigoplus_{i=1}^n \bU(-e_i) \ar[r]^\simeq & \bigoplus_{i=1}^n \bU(-e_i) 
}
$$
\end{small}
where $\varphi$ denotes the composition of the surjection
$$
\bigoplus_{i=1}^n \bU(-e_i) \oplus \fg \otimes \bU(-1) \longrightarrow \Der_\bU(-\log(\Delta))
$$
with the natural injection $\Der_\bU(-\log(\Delta)) \hookrightarrow \fg \otimes \bU$, whose cokernel is given by the restricted Jacobian ideal $J_D$, up to a degree shift.
Applying the snake lemma to the previous diagram induces the following exact sequence
$$
0 \to \ker (\bad) \to \bigoplus_{i=1}^n \bU(-e_i-1) \stackrel{\beta}{\to} \bigoplus_{i=1}^n \bU(-e_i) \to \coker (\bad) \to J_D(d-1) \to 0.
$$
Furthermore, in light of Lemma \ref{lemma-kerad} and the first injection of the previous sequence, we have
\begin{equation} \label{ker ad}
    \ker (\bad) \simeq \bigoplus_{i=1}^n \bU(-e_i-1),
\end{equation} and hence $\beta = 0$. This implies that
\begin{align*}
    \Der_\bU(-\log(\Delta)) =& \IM(\varphi) \simeq \bA \oplus \Der_\bU(-\log(\Delta))^G\simeq\\ &\simeq \Der_\bU(-\log(\Delta))_1 \oplus \Der_\bU(-\log(\Delta))^G. \qedhere
\end{align*}
\end{proof}

Now we are in position to prove our main result.
\begin{proof}[Proof of Theorem \ref{main simply laced}]
Using Theorem \ref{diamogli un nome} and the isomorphism \eqref{exponents and invariant derivations}, we get a decomposition 
\[
\Der_\bU(-\log(\Delta)) \simeq \bigoplus_{i=1}^n \bU(-e_i) \oplus \bA,
\]
Recall that the module $\bA$ is defined as the image of $\bad : \fg\otimes \bU(-1) \to \fg\otimes \bU$. Using \eqref{ker ad} we get that $\bA$ fits into:
\[
0\to \bigoplus_{i=1}^n \bU(-e_i-1) \xrightarrow{\nu} \fg\otimes \bU(-1) \to \bA \to 0.\]
\end{proof}
In terms of sheaves, the previous theorem is translated into the next result.
\begin{corollary}
Let $G$ be a simple linear algebraic $\bk$-group of simply laced type and $\fg$ its lie algebra.
Let $e_1,\ldots,e_n$ be the exponents of $\fg$ and let $D=\VV(\Delta)$ be the adjoint discriminant hypersurface. Then:
\[
  \hat \cT_\PP\langle D \rangle 
\simeq \bigoplus_{i=1}^n \bU(1-e_i) \oplus \hat \cT_\PP^{\mathbf{ad}},
\]
where the sheaf $\hat \cT_\PP^{\mathbf{ad}}$ is the image of $\bad : \fg\otimes \cO_{\PP(\fg)} \to \fg\otimes \cO_{\PP(\fg)}(1)$ and fits into:
\[
0\to \bigoplus_{i=1}^n \cO_{\PP(\fg)}(-e_i) \xrightarrow{\nu} \fg\otimes \cO_{\PP(\fg)} \to \hat \cT_\PP^{\mathbf{ad}} \to 0.
\]
\end{corollary}

\bibliographystyle{amsalpha.v2}
\bibliography{discriminants}

\end{document}